\documentclass[11pt,
  twoside]{amsart}

\usepackage{dsfont} 

\usepackage{wrapfig,lmodern}
\usepackage{graphicx}
\usepackage{color}
\usepackage{amsmath}
\usepackage{amssymb}
\usepackage{amsfonts}
\usepackage{latexsym, amsthm, mathrsfs}
\usepackage{hyperref, breakurl}
\usepackage[latin1]{inputenc}
\usepackage[T1]{fontenc} 
\usepackage{lmodern} 
\usepackage{fancybox}
\usepackage{amscd}
\usepackage[all]{xy}
\usepackage[shortlabels]{enumitem} 
\usepackage{tikz-qtree}
\usepackage{rotating}
\usepackage{stmaryrd}
\usepackage{pifont} 

\numberwithin{equation}{section}
\newtheorem{theorem}{Theorem}[section]
\newtheorem{lemma}[theorem]{Lemma}

\theoremstyle{definition} 
\newtheorem{definition}[theorem]{Definition}

\newtheorem{example}[theorem]{Example}
\theoremstyle{plain} 
\newtheorem{question}[theorem]{Question}

\newtheorem{proposition}[theorem]{Proposition}
\newtheorem{corollary}[theorem]{Corollary}
\newtheorem{fact}[theorem]{Fact}

\newtheorem{claim}[theorem]{Claim}

\newtheorem*{maintheorem*}{Main Theorem}
\newtheorem*{conjecture*}{Conjecture}
\newtheorem*{theorem*}{Theorem}
\newtheorem*{proposition*}{Proposition}
\newtheorem*{corollary*}{Corollary}

\newtheorem{remark}[theorem]{Remark} 

\newtheorem{observation}[theorem]{Observation}
\theoremstyle{remark}  
\newtheorem*{remarks*}{Remarks}
\newtheorem*{remark*}{Remark}
\newtheorem*{claim*}{Claim}

\newtheoremstyle{mystyle}
  {3pt}{3pt}{\itshape}{}{\bfseries}{}{.5em}
  {\thmname{#1}\thmnumber{ #2}\thmnote{. #3}}
\theoremstyle{mystyle}
\newtheorem*{maintheoremnew*}{Main Theorem}
\newtheorem*{conjecturenew*}{Conjecture}
\newtheorem*{theoremnew*}{Theorem}
\newtheorem*{propositionnew*}{Proposition}

\newcommand{\nc}{\newcommand}
\nc{\nothing}[1]{}
\nothing{   
\nc{\dom}{{\rm dom}}
\nc{\card}{{\rm card}}
\nc{\lh}{{\rm lh}}
\nc{\lgg}{{\rm lh}}
\nc{\rge}{\mbox{\rm range}}
\nc{\cf}{{\rm cf}}
nc{\nex}{\mbox{\rm next}}
\nc{\uhr}{\restriction}
\nc{\supt}{{\rm supt}}
\nc{\supp}{{\rm supp}}
\nc{\Lim}{{\rm Lim}}
\nc{\Leb}{{\rm Leb}}
\nc{\modd}{{\rm mod}}
\nc{\RO}{{\rm RO}}
\nc{\prob}{{\rm Prob}}
}

\nc{\Ord}{{\rm On}}
\nc{\nco}{\DeclareMathOperator}
\nco{\cc}{cc}
\nco{\rk}{rk}
\nco{\llower}{lower}
\nco{\order}{o}
\nco{\Nm}{Nm}
\nco{\ppower}{pp}
\nco{\pcf}{pcf} 
\nco{\tcf}{tcf} 
\nco{\tlim}{tlim} 
\nco{\limtext}{lim} 
\nco{\prodt}{{\textstyle \prod}}
\nco{\symdiff}{\triangle}
\nco{\dom}{dom}
\nco{\card}{card}
\nco{\lh}{lh}
\nco{\lt}{lt}
\nco{\lgg}{lh}
\nco{\hgt}{ht}
\nco{\rge}{range}
\nco{\otp}{otp}
\nco{\trunk}{tr}
\nco{\cf}{cf}
\nco{\nex}{next}
\nc{\uhr}{\restriction}
\nco{\reduction}{red}
\nco{\supt}{supt}
\nco{\supp}{supp}
\nco{\Lim}{Lim}
\nco{\Leb}{Leb}
\nco{\modd}{mod}
\nco{\invariant}{inv}
\nco{\invzwei}{{inv}^2_\kappa}
\nco{\inveins}{{inv}^1_\kappa}

\newenvironment{myrules}
{\begin{list}{}
{
 \setlength{\leftmargin}{1.1cm}
 \setlength{\labelwidth}{1.1cm}
 \setlength{\labelsep}{0.2cm}
 \setlength{\parsep}{0.7ex plus 0.2ex minus 0.1 ex}
 \setlength{\itemsep}{0.5ex plus 0.2 ex minus 0ex}
}}{\end{list}}

\newcount\skewfactor

\def\mathunderaccent#1#2 {\let\theaccent#1\skewfactor#2
\mathpalette\putaccentunder}
\def\putaccentunder#1#2{\oalign{$#1#2$\crcr\hidewidth
\vbox to.2ex{\hbox{$#1\skew\skewfactor\theaccent{}$}\vss}\hidewidth}}
\def\name{\mathunderaccent\tilde-3 }

\nc{\nname}{\name}

\newcommand{\la}{\langle}
\newcommand{\ra}{\rangle}

\nc{\lebesgue}{\ensuremath{{\cal N}}}
\nc{\nulll}{\ensuremath{{\cal N}}}
\nc{\ksigma}{\ensuremath{{\bf K}_\sigma}}

\nc{\ga}{\ensuremath{\frak a}}
\nc{\AAA}{{\cal A}}   
\nc{\gc}{\ensuremath{\frak c}}
\nc{\gn}{\ensuremath{\frak n}}
\nc{\gh}{\ensuremath{\frak h}}
\nc{\gd}{\ensuremath{\frak d}}
\nc{\gb}{\ensuremath{\frak b}}
\nc{\gro}{\ensuremath{\frak g}}
\nc{\gu}{\ensuremath{\frak u}}
\nc{\gr}{\ensuremath{\frak r}}
\nc{\gt}{\ensuremath{\frak t}}
\nc{\fff}{\ensuremath{\frak f}}
\nc{\gm}{\ensuremath{\mathfrak{mcf}}}
\nc{\gge}{\ensuremath{\mathfrak e}}
\nc{\cfupro}{\ensuremath{\cf(\upro)}}
\nc{\cfvpro}{\ensuremath{\cf(\vpro)}}
\nc{\gp}{\ensuremath{\frak p}}
\nc{\gs}{\ensuremath{\frak s}}

\nc{\gothe}{\mathfrak{e}}
\nc{\add}{\mbox{\ensuremath{{\rm add}}}}
\nc{\cov}[1]{\mbox{\ensuremath{{\rm cov}(#1)}}}
\nc{\unif}[1]{\mbox{\ensuremath{{\rm unif}(#1)}}}
\nc{\cof}[1]{{\mbox{\ensuremath{\rm cof}(#1)}}}

\nc{\addd}[2]{\mbox{\ensuremath{{\rm add}^{#1}(#2)}}}   
\nc{\covv}[2]{\mbox{\ensuremath{{\rm cov}^{#1}(#2)}}}   
\nc{\uniff}[2]{\mbox{\ensuremath{{\rm unif}^{#1}(#2)}}} 
\nc{\coff}[2]{{\mbox{\ensuremath{\rm cof}^{#1}(#2)}}}

\nc{\cd}{Cicho\'n's Diagram}

\nc{\AC}{\mbox{\sf AC}}
\nc{\MA}{\mbox{\sf MA}}
\nc{\PFA}{\mbox{\sf PFA}}
\nc{\OCA}{\mbox{\sf OCA}}
\nc{\GCH}{\mbox{\sf GCH}}
\nc{\CH}{\mbox{\sf CH}}
\nc{\ZFC}{\mbox{\sf ZFC}}
\nc{\sch}{\mbox{\sf SCH}}
\nc{\ZF}{\mbox{\sf ZF}}
\nc{\NCF}{\mbox{\sf NCF}} 
\nc{\FD}{\mbox{\sf FD}}   

\nc{\eps}{\varepsilon}
\renewcommand{\epsilon}{\varepsilon}

\newcommand{\cal}{\mathcal}

\nco{\tnode}{term}
\nco{\acc}{acc}
\nco{\last}{last}
\nc{\mup}{m_{\rm up}}
\nc{\mdn}{m_{\rm dn}}
\nco{\may}{may}
\nco{\aver}{av} 
\nco{\norm}{nor} 
\nco{\val}{val} 
\nco{\dis}{dis} 
\nco{\basis}{basis}
\nco{\pos}{pos}
\nco{\spec}{spec}
\nc{\err}{\mbox{err}}
\nc{\eee}{\mbox{e}}
\nco{\Expect}{Exp}
\nco{\rt}{rt}
\nco{\tr}{tr} 
\nco{\suc}{succ}
\nco{\osucc}{osucc} 
\nco{\halv}{h}
\nco{\Add}{Add}
\nco{\Cov}{Cov}
\nco{\Unif}{Unif}
\nco{\Cof}{Cof}
\nco{\htt}{ht}
\nco{\cl}{cl}
\nco{\Levy}{Levy}
\nco{\set}{set}
\nc{\bbforcing}{\mathbb A}
\nc{\itername}{\mathfrak q}
\nc{\iterp}{\mathfrak p}
\nc{\iterq}{\mathfrak q}
\nc{\invcm}{\rm inv_{cm}}
\nc{\invcf}{\rm inv_{cf}}
\nc{\invgm}{\rm inv_{gm}}

\nc{\Q}{\mathbb Q}
\nc{\subsetsim}{\underset{\raise0.6em\hbox{$\sim$}}{\subset}}
\newcommand{\subsim}{\underset{\raise20pt\hbox{$\rightarrow$}}{\rightarrow}}
\newcommand{\ssim}{\overset{\raise-40pt\hbox{$\leftarrow$}}{\subsim}}
\nc{\rest}{\restriction}
\nc{\bF}{\mathbb F}
\nc{\bH}{\mathbb H}
\nc{\R}{\mathbb R}
\nc{\M}{\mathbb M}
\nc{\bQ}{\mathbb Q}
\nc{\bS}{\mathbb S}

\nc{\qaam}{{\mathbb Q}^{\rm a, amoeba}_\kappa}
\nc{\bP}{\mathbb P}
\nc{\bG}{\mathbf G}
\nc{\cJ}{\mathcal J}

\nc{\cF}{\mathscr F}
\nc{\cB}{\mathcal B}

\nc{\bV}{\mathbf{V}}
\nc{\cE}{\mathcal E}
\nc{\cP}{\mathscr P}
\nc{\cU}{\mathcal U}
\nc{\cV}{\mathcal V}
\nc{\cW}{\mathcal W}
\nc{\cC}{\mathcal C}
\nc{\cT}{\mathcal T}
\nc{\cX}{\mathcal X}
\nc{\cD}{\mathcal D}
\nc{\cO}{\mathscr O}
\nc{\cS}{\mathcal S}
\nc{\cA}{\mathcal A}
\nc{\cG}{\mathcal G}
\nc{\cI}{\mathcal I}
\nc{\ba}{\mathbf{a}}

\renewcommand{\phi}{\varphi}

\newcommand{\bb}{\mathfrak{b}}

\newcommand{\cohen}{\poset{C}}
\newcommand{\bC}{\poset{C}}
\newcommand{\conc}{\smallfrown}
\newcommand{\concat}{{}^\smallfrown} 

\newcommand{\club}{\mathrm{Club}}

\newcommand{\ideal}{\mathcal}

\newcommand{\force}{\Vdash}

\newcommand{\On}{\text{On}}

\newcommand{\poset}{\mathbb}

\newcommand{\restric}{{\upharpoonright}}

\newcommand{\silver}{\poset{V}}
\newcommand{\mathias}{\poset{MA}}

\newcommand{\SSigma}{\mathbf{\Sigma}}

\newcommand{\spl}{\poset{SP}}
\newcommand{\such}{\: : \:}

\nc{\comment}[1]{#1}

\nco{\codes}{code}
\nco{\taille}{taille}
\nco{\ter}{ter}

\nco{\Fn}{Fn}
\nco{\borel}{Bor}

\nco{\height}{ht}
\nco{\level}{Lev}
\nco{\levy}{Coll}

\nco{\ot}{ot}
\nco{\rank}{Rank}
\nco{\splitting}{Split}
\nco{\splitlevel}{ns}
\nco{\stem}{stem}
\nco{\successor}{succ}
\nco{\Succ}{Suc}
\nco{\splsuc}{splsuc}
\nco{\Lev}{Lev}
\nco{\en}{en}
\nco{\term}{Term}

\newcommand{\xmark}{\ding{55}}
\newcommand{\cmark}{\ding{51}}

\begin{document}

\title[Trees with positive upper density, compiled \today]{Mathias and Silver forcing parametrized by density}
\author{Giorgio Laguzzi, Heike Mildenberger, Brendan Stuber-Rousselle} 
\begin{abstract}
We define and investigate versions of Silver and Mathias forcing with respect to lower and upper density. We focus on properness, Axiom A, chain conditions, preservation of cardinals and adding Cohen reals. We find rough forcings that collapse $2^\omega$ to $\omega$, while others are surprisingly gentle. We also study connections between regularity properties induced  by  these parametrized forcing  notions and the Baire property.
\end{abstract}
\date{Feb.\ 9, 2021}
\maketitle

\section{Introduction}

Forcings consisting of tree conditions are the \emph{dramatis personae} extensively studied in descriptive set theory and set theory of the reals. The original interest on this type of forcings was mainly due to their applications in resolving problems of independence risen from the study of cardinal characteristics, infinitary combinatorics and regularity properties. Through the years, the study of the combinatorial properties of tree-forcings has been intensified and has been interesting on its own.

From the forcing point of view, trees can be thought as conditions whose stem decides a finite fragment of the generic real, and the rest of the tree above describes the possible paths of the generic real. In this sense, Cohen forcing can be understood as the simplest tree-forcing notion, whose conditions decide a finite fragment of the generic and then leave all paths extending the stem possible. In general, the fatter the tree, the more freedom the path of the generic real; the slimmer the tree, the more restrictive the conditions on the path of the generic real.  Under this point of view a Sacks tree can be seen as the other extreme, since in this case the tree can be shrunk as much as desired, and the only requirement is to keep perfectness. 

Many other tree-forcings in between have been introduced and extensively studied, among them Silver and Mathias trees\footnote{The nodes of the Mathias (or Silver) tree $p$ are all $t \in 2^{<\omega}$ that are initial segments of one of the possible generic branches that are compatible with a Mathias condition $(s,A)$ (or  
a Silver condition $f_p$).}. In this paper we focus on studying some variants where the set of nodes above the stem are governed by restrictions imposed on the density of the set of splitting nodes. Imposing fatness or slimness conditions though, may result in a non-proper forcing.

Each condition in Mathias forcing and in Silver forcing, when conceived as a tree, comes with an infinite set $A_p$ such that any node $t$ of $p$ whose length is in $ A_p$ is a splitting node. Also other tree forcings, like Sacks or Miller, can be thinned out to the subsets of those conditions who come with such a set $A_p$ of overall splitting levels.  In Bukovsk\'{y}-Namba forcing the conditions with uniform splitting levels are even dense, see e.g.\ \cite[Theorem 2.2]{Bukovsky_Coplakova}. 

We let the $A_p$ range over prescribed families that are sets of positive lower or upper density in $\omega$. There is some work on forcings of this type. Grigorieff \cite{Grigorieff} parametrised Silver forcing by having the $A_p$ range over a $P$-point and Halbeisen \cite[Ch 24]{Halbeisen} generalised this to $P$-families.  Mathias forcing with a Ramsey ultrafilter \cite{Mathias:happy} is a versatile notion of forcing. Farah and Zapletal \cite[Ch.\ 9]{Farah2006} used 
the coideal of the density zero ideal as a reservoir
for the infinite component of Mathias conditions.\\

In Sections 2 to 5 we investigate whether the variants we introduce of Silver and of Mathias forcing  are proper at all. We show that Mathias with lower density $\geq \varepsilon$ is a disjoint union of $\sigma$-centered forcings (see Proposition \ref{sigma_centered}) and that Silver forcing with positive lower density collapses $2^\omega$ to $\omega$ (see Theorem \ref{collapse}). We also show that the lower density is far from the notion of a measure (see Proposition \ref{continuum_antichains}). This result fits in the framework presented in \cite{MR4054777}.

In Section 6 and 7 we are concerned with regularity properties of these forcings. Regularity properties of tree-forcings by themselves are a well studied field in mathematics. The following table illustrates which of the most popular regularity property of a subset of $2^\omega$ correspond to which tree-forcing $\poset P$.
\vspace{3mm}
\begin{center}
\begin{tabular}{|c|c|}\hline
 Tree-forcing $\mathbb{P}$  & $\mathbb{P}$-measurable \\ \hline
Cohen  $\mathbb{C}$   & Baire property \\ \hline
Random $\mathbb{B}$ & Lebesgue  measurable \\ \hline
Silver $\mathbb{V}$  & completely  doughnut \cite{PriscoH00} \\ \hline
Sacks $\mathbb{S}$  & Marczewski set \cite{Szpilrajn1935} \\ \hline
Mathias $\mathias$ & completely Ramsey \\ \hline
\end{tabular}
\end{center}
 \vspace{3mm}
In many cases adding a Cohen real is enough to establish a dependence between Baire property and the measurability given by the tree-forcing. This was made explicit in \cite[Proposition 3.1.]{Laguzzi_Stuber}. Proposition \ref{gamma p implies gamma q} is an improvement of this result.

In Section 8 we construct a model in which all $\On^\omega$-definable sets are $\silver^+_\varepsilon$-measurable but $\Sigma^1_2(\cohen)$ fails. In particular $\Sigma^1_2(\silver^+_\varepsilon) \Rightarrow \Sigma^1_2(\cohen)$ fails.

In Section 9 we return to a question we asked in \cite{fat}, namely: Does the splitting forcing $\mathbb{SP}$ have the Sacks property? We give a partial negative answer in Theorem~\ref{not_Sacks}.  Independently, Jonathan Schilhan found a complete negative answer to our question \cite{Schilhan}.

In the remainder of this introduction, we set up our notation. 

\begin{definition}\label{1.1}
Let $X$ be non-empty.
\begin{enumerate}[label=(\alph*)]
\item We let $X^{<\omega} = \{s \such (\exists n < \omega )(s \colon n \to X)\}$.
The set $X^{<\omega}$ is partially ordered by the initial segment relation $\trianglelefteq$, namely  $s \trianglelefteq t$ if $s = t \restric \dom(s)$. We use  $\triangleleft$ for the strict relation.

\item A set $p \subseteq X^{<\omega}$ is called a \emph{tree} if it is closed under initial
 segments, i.e. $(t \in p \wedge s \trianglelefteq t )\rightarrow s \in p$.
The elements of $p$ are called \emph{nodes}.

\item  A node is called a \emph{splitting node} if it has at least two immediate successors in $p$. We write  $\splitting(p)$ for the set of splitting nodes of $p$.
\item  A tree p is called \emph{perfect}
  if for every $s \in p$ there is a splitting node $t \trianglerighteq s$.

\item For $t \in p$, we write $\splsuc(t)$ for the shortest splitting node extending $t$. When $t$ is splitting, then $\splsuc(t)=t$.
\item The \emph{stem of $p$}, short $\stem(p)$, is the $\trianglelefteq$-least splitting node of $p$ if it exists.

\item  For $n < \omega $ we let $\splitting_n (p)$ consist of all splitting nodes $t$ in $p$ such that there are  exactly $n$  splitting nodes preceding $t$ i.e.,  $ \stem(p) = t_0 \triangleleft t_1 \triangleleft  \dots  \triangleleft t_{n-1} \triangleleft t$, in particular $\splitting_0(p)= \{\stem (p) \}$. Analogously, we define $\splitting_{\leq n}(p)$.
\item For $n < \omega$, let $\Lev_n (p) := \{ t \in p\such |t| = n\}$.
\item For $t \in p$ we let $p \restric t = \{s \in p \such s\trianglelefteq t \vee t \triangleleft s\}$.  

\item Let $p\subseteq X^{<\omega}$ be a tree such that for any $t \in p$ and any $n$ there is $s \in p$, 
$t \trianglelefteq s$ such that $|s|>n$. The body or rump of a tree $p$, short $[p]$, is
 the set $\{f \in X^\omega \such (\forall n)( f \restric n\in p)\}$. 
\end{enumerate}
\end{definition}

\begin{definition}\label{trees}
A partial ordering $(\poset P,\leq )$ is called a \emph{tree-forcing}, if there is a non-empty set $X$ such that 
\begin{enumerate}
\item[-] All conditions $p \in \poset P$ are perfect trees on $X^{<\omega}$.
\item[-] For all $p\in \poset P$ and $t\in p$ the restriction $p\restric t $ is again a condition in $\poset P$.
\item[-] The partial order is the inclusion i.e., $q\leq p$ iff $q\subseteq p$.  
\end{enumerate}
\end{definition}

We are interested in tree-forcings $\mathbb{P}$ defined over $2^{<\omega}$ or $\omega^{<\omega}$ with the following additional property. 

\begin{definition}\label{uniformly_splitting}
Let $p$ be a perfect tree defined over $2^{<\omega}$ or $\omega^{<\omega}$. The tree $p$ is called \emph{uniformly splitting}, if there is an infinite set $A_p \in [\omega]^\omega$ such that 
$$\forall s\in p(|s| \in A_p \Leftrightarrow s\in \splitting (p)).$$
A tree-forcing $\poset P$ is called \emph{uniformly splitting tree-forcing}, if all conditions $p\in \poset P$ are uniformly splitting.
\end{definition}

In this paper we study two well-known examples of uniformly splitting tree forcings: The Mathias forcing $\mathias$ and the Silver forcing $\silver$.

\begin{example}
\begin{enumerate}
\item The Mathias forcing  is given by conditions $(s,A) \in \mathias$ if $s \in [\omega]^{<\omega}$ and $A \in [\omega]^\omega$ and $\max(s) < \min(A)$.
$(t,B) \leq (s, A)$ if $t \setminus s \subseteq A$, $t \supseteq s$, and  $B \subseteq A$.
Now a Mathias tree $p = p(s,A)$ is given by $p \subseteq  2^{<\omega}$ and
$t \in p$ if $t$ is the characteristic function of
$s \cup s'$ for some finite $s ' \subseteq A$.
Of course, now $A_p = A$, and stronger conditions correspond to subtrees. All trees are perfect and restrictions $p \restric t$ are again conditions.

\item A condition $f$ is a Silver condition if there is an infinite set $A$ such that $f : \omega \setminus A \rightarrow 2$.
So we get a Silver tree $p \subseteq 2^{<\omega}$ by letting 
\[
t \in p \leftrightarrow (\forall n \in |t| \setminus A) (t(n) = f(n)). 
\]
Then $A_p = A$. Stronger conditions are extensions and again correspond to subtrees.
\end{enumerate}
\end{example}

\begin{definition}
For a set $A\subseteq \omega$ we define the upper density $d^+(A)$ and lower density $d^-(A)$ of $A$ via:
\begin{enumerate}
\item $d^+(A):= \limsup\limits_{n \rightarrow \infty} \frac{|A\cap n|}{n}$, 

\item $d^-(A):= \liminf\limits_{n \rightarrow \infty} \frac{|A\cap n|}{n}.$
\end{enumerate}
\end{definition}

\begin{definition}\label{erweitert}
Let $\mathbb{P}$ be a uniformly splitting tree-forcing defined over $2^{<\omega}$ or $\omega^{<\omega}$. For $\varepsilon\in (0,1]$ we define two subforcings $\mathbb{P}^+_\varepsilon$ and $\mathbb{P}^-_\varepsilon$, and we define the upper and lower positive density versions:
\begin{enumerate}
\item $p\in \mathbb{P}^+_\varepsilon$ if $p\in \mathbb{P}$ and $A_p$ has upper density $\geq \varepsilon$.
\item $p\in \mathbb{P}^-_\varepsilon$ if $p\in \mathbb{P}$ and $A_p$ has lower density $\geq \varepsilon$.
\item $p\in \mathbb{P}^+$ if $p\in \mathbb{P}$ and $A_p$ has upper density $>0$.
\item $p\in \mathbb{P}^-$ if $p\in \mathbb{P}$ and $A_p$ has lower density $>0$.
\end{enumerate}
In all forcing orders, a condition $q$ is stronger than $p$ iff $q$ is a subset of $p$.   
\end{definition}


We focus our attention on $\mathias^+_\eps$, $\mathias^-_\eps$, $\mathias^+$, $\mathias^-$ and the same for Silver. We order our investigation now in pairs, according to the density requirement. Some steps work also for general $\poset P$.

\section{Upper density $\geq \varepsilon$}\label{upper_density_geq}

\begin{definition}\label{Axiom_A}
  A notion of forcing $(\bP,\leq)$ has \emph{Axiom A} if there are partial order relations $\la \leq_n \such n < \omega \ra$ such that
  \begin{myrules}
  \item[(a)] $q \leq_{n+1} p$ implies $q\leq_n p$ , $q\leq_0 p$ implies
    $q \leq p$,
  \item[(b)] If $\la p_n \such n < \omega \ra$ is a fusion sequence, i.e.,  a sequence such that
    for any $n$, $p_{n+1} \leq_n p_n$, then there is a lower bound $p \in \bP$, $p \leq_n p_n$. 
  \item[(c)] For any maximal antichain $A$ in $\bP$ and and $n\in \omega$ and any $p \in \bP$ there is $q \leq_n p$ such that only countably many elements of $A$ are compatible with $q$. Equivalently, for any open dense set $D$ and any $n$, $p$, there is a countable set $E_{p}$ of conditions in $D$ and $q\leq_n p$ such that $E_{p}$ is predense below $q$.

  A notion of forcing $(\bP,\leq)$ has \emph{strong Axiom A} if the set of compatible elements in (c) is even finite. 
 \end{myrules}
\end{definition}

Axiom A entails properness and
strong Axiom A implies ${}^\omega \omega$-bounding (see, e.g., \cite[Theorem 2.1.4,
  Cor~2.1.12]{RoSh:470}).

\begin{remark}\label{sequence}
Let $p\in\mathbb{P^+_\varepsilon}$. We define an increasing sequence $\langle k^p_n \in \omega \such n\in \omega \rangle$ as follows:
\begin{align*}
& n=0 : \text{ We put }k^p_0:=  \min(A_p) ,\\
& n>0 : \text{ We let } k^p_{n} :=\min \{ k \in \omega \such ( k > k^p_{n-1}   \wedge |A_{p}\cap k |\geq k(\varepsilon - 2^{-n}) ) \} .
\end{align*}

Such a sequence has the following property for $n\in \omega$:
\[
\frac{|A_{p}\cap k^p_n|}{k^p_n}  \geq \varepsilon-2^{-n},
\] 
and therefore witnesses $d^+(A_p)\geq \varepsilon$.
\end{remark}
We use the sequences $\langle k^p_n \in \omega \such n\in \omega \rangle$ to define a \emph{stronger-$n$}-relation $\leq_n$ on $\mathbb{P^+_\varepsilon}$.
\begin{definition}\label{stronger_n}
We define a decreasing sequence of partial order relations $\langle \leq_n \such n\in \omega \rangle$ on $\mathbb{P}^+_\varepsilon$ as follows:
\[
q\leq_n p \Leftrightarrow q\leq p \wedge k^q_n=k^p_n \wedge A_q \cap k^q_n = A_p \cap k^p_n.
\]
\end{definition}
Observe that given two conditions satisfying $q\leq_n p$ we must have $k^q_i = k^p_i$ for $i\leq n$.
\begin{fact}\label{fusion1}
Let $\poset P \in \{\mathias,\silver \}$. Let $\langle q_n \such n\in \omega \rangle$ be a fusion sequence in $\mathbb{P}^+_\varepsilon$. Then, fusions exist. Especially, $q = \bigcap q_n$ is a condition in $\mathbb{P}^+_\varepsilon$.
\end{fact}

\subsection{Silver forcing with upper density $\varepsilon>0$}
We quickly establish that the Silver forcing with positive upper density $\varepsilon$ has strong Axiom A and thus is a proper forcing that does not add unbounded reals. The proof is a straightforward generalization of the standard case.

\begin{lemma}\label{properandbounding}
The forcing $\silver^+_\varepsilon$ has strong Axiom A.
\end{lemma}

\begin{proof}
We take the partial order relations $\leq_n$ as defined in Definition \ref{stronger_n}. It is easy to see that the requirements (a) and (b) from Definition \ref{Axiom_A} are fulfilled. We make sure that the strong version of (c) can be fulfilled as well. So, fix an open dense set $D\subseteq \mathbb{V}^+_\varepsilon$, $n\in \omega$ and a condition $p\in \mathbb{V}^+_\varepsilon$. We have to find a finite set $E_p \subseteq D$ and  a condition $q\leq_n p$ such that $E_p$ is predense below $q$.\\
To this end, let $\langle k^p_m \such m<\omega \rangle$ be the sequence associated to the condition $p$ as given in Remark \ref{sequence}.  Let $k:= |A_{p} \cap k_n^{p} |$ denote the number of splitting levels up to $k_n^{p}$ and $\{ t_i \in \splitting (p)  \such i < 2^{k}\}$ enumerate all splitting nodes of length exactly $k_n^{p}$. We construct a decreasing sequence $p =: q_0 \geq q_1 \geq \dots \geq q_{2^{k}} =: q$ such that $q_{i+1} \restric t_i \in D ,i< 2^{k}$. Fix $i<2^k$. Now, pick $p_i \leq q_{i} \restric t_i$ in $D$ and copy $p_i$ into $q_i$ above level $k_n^{p}$, more precisely:
\[
q_{i+1} := \{s \in q_i \such (|s| \leq k_n^{p} \vee (|s| > k_n^{p} \wedge \exists t \in p_i (k_n^{p} \leq m < |s| \rightarrow s(m) = t(m) ) ) ) \}.
\]
It is clear that the resulting tree $q$ satisfies $q \restric t \in D$, whenever $t \in q$ is of length at least $k_n^{p}$. Moreover, since we have pruned the tree $q$ only above $k_n^{p}$ we also made sure that
\[
k^{p}_n = k^{q}_n \wedge A_{p} \cap k^{p}_n = A_{q} \cap k^{p}_n,
\]
especially $q\leq_n p$. For the finite set $E_p$ we can simply put $E_p:= \{q \restric t \such t\in \level_{k^p_n}(q) \}$. This completes the proof.
\end{proof}

\begin{corollary}
The forcing $\silver^+_\varepsilon$ is proper and $\omega^\omega$-bounding for each $\varepsilon\in (0,1]$.
\end{corollary}

\subsection{Mathias forcing with upper density $\varepsilon>0$}
In the following, we investigate the differences between $\mathias^+_\varepsilon$ and the classical Mathias forcing $\mathias$ and Mathias forcing $\mathias\ideal (\ideal F)$ with respect to a filter $\ideal F$. The decisive difference to the classical forcing is that $\mathias^+_\varepsilon$ adds Cohens. From $\mathias(\ideal F)$ it already differs by the fact that the set $\{ A\subseteq \omega \such d^+(A) \geq \varepsilon \}$ is not closed under finite intersection and thus not a filter. Note that this does not yet imply that the two forcing are not forcing equivalent.
\begin{definition}
Let $\ideal F$ be a filter over $\omega$. The partial order $\mathias(\ideal F)$ consists of all $p\in \mathias$ such that the set of splitting levels $A_p$ is an element of the filter $\ideal F$. $\mathias(\ideal F)$ is ordered by inclusion. 
\end{definition}
Remember that a filter $\ideal F$ over $\omega$ is called \emph{Canjar}, if the corresponding Mathias forcing $\mathias (\ideal F)$ does not add dominating reals. See for instance \cite{canjar:mathias}, where Michael Canjar  constructed an ultrafilter $\ideal U$ under the assumption $\mathfrak{d}= \mathfrak{c}$ such that $\mathias(\ideal U)$ does not add dominating reals. 
The following lemma shows that $\mathias^+_\varepsilon$ cannot be equivalent to $\mathias(\ideal F)$, for $\ideal F$ Canjar.
\begin{lemma}\label{dominating}
$\mathias^+_\varepsilon$ adds dominating reals.
\end{lemma}
\begin{proof}
Let $G$ be $\mathias^+_\varepsilon$-generic over $V$ and let $x_G = \bigcup\{ s \such \exists A ((s,A) \in G)\}$ denote the the $\mathias^+_\varepsilon$-generic real. Observe that $x_G$ has upper density $\varepsilon$ in the generic extension $V[G]$. Therefore, the following function is well defined  in $V[G]$:
\[
{f}(n):= \min\left\{k \such \frac{|x_G\cap k|}{k} \geq \varepsilon - 2^{-n} \right\}.
\]
Remember that in Remark \ref{sequence} we assigned to each condition $p\in \mathias^+_\varepsilon$ a sequence $\la k^p_n \in \omega \such n<\omega \ra$. That $f$ is dominating now follows from the following two facts:
\begin{enumerate}
\item[(i)] $\forall p\in \mathias^+_\varepsilon, p \force \forall n<\omega ( f (n) \geq k^p_n ) $,
\item[(ii)] $ \forall g\in \omega^\omega \forall p \in \mathias^+_\varepsilon \exists q\leq p, q\force \forall^\infty	n (k^q_n \geq g(n)) $.
\end{enumerate}
\end{proof}

There is a more general reason why $\mathias^+_\varepsilon$ and $\mathias(\ideal F)$ for any filter $\ideal F$ cannot be forcing equivalent. Since $\ideal F$ is a filter, any two conditions $p,q\in\mathias (\ideal F)$ with the same stem are compatible and in particular $\mathias^+_\varepsilon$ is $\sigma$-centered.

\begin{lemma}
$\mathias^+_\varepsilon$ does not satisfy the countable chain condition. 
\end{lemma}

\begin{proof}
Without loss of generality let $\varepsilon= 1$. We construct an injective function $f : 2^{\omega} \rightarrow [\omega]^\omega$ such that each $x\in 2^{\omega}$ with $|x^{-1}(1)|=\omega $ get mapped to a set with upper density $1$. To this end we define $f$ together with an auxiliary function $g:\omega \rightarrow \omega$ 
\begin{align*}
g(0) &:= 0 \\
g(i+1) &:= 2^{g(i)}\\ 
n \in f(x) & : \Leftrightarrow \text{ there is } i\in \omega , x(i)=1 \text{ and } n\in [g(i),g(i+1)) 
\end{align*}
Now, choose an almost disjoint family $\{A_\alpha \in[\omega]^\omega  \such \alpha<\mathfrak{c}\}$ of size continuum and let $x_\alpha$ denote the corresponding characteristic function i.e., $x_\alpha(i)=1 $ iff $i\in A_\alpha$. Then, $\{f(x_\alpha) \such \alpha<\mathfrak{c}\}$ is an almost disjoint family consisting of sets of upper density $1$. In particular, the conditions $\{(\la\ra , f(x_\alpha)) \such \alpha<\mathfrak{c}  \}$ form an antichain in $\mathias^+_1$ of size the continuum.
\end{proof}

Now we turn our attention to a comparison with the classical Mathias forcing $\mathias$. Usually the first step in showing that a given forcing satisfies \emph{pure decision} (compare \cite[Lemma 7.4.5]{BJ}) is to show that the forcing satisfies \emph{quasi pure decision} (compare \cite[Lemma 7.4.6]{BJ}). We will see that although $\mathias^+_\varepsilon$ fails to satisfy pure decision it still satisfies quasi pure decision.

\begin{lemma}\label{quasi_pure_decision}
$\mathias^+_\varepsilon$ satisfies quasi pure decision i.e., given a condition $p=(s,A)\in \mathias^+_\varepsilon$ and an open dense set $D\subseteq \mathias^+_\varepsilon$ there is $B\subseteq A$ such that the following holds:
\[
\text{If there is } (t,C)\leq (s,A) \text{ and } (t,C)\in D, \text{ then } (t,B\setminus (\max(t)+1))\in D.
\]
\end{lemma}

\begin{proof}
The proof is a straightforward generalization of \cite[Lemma 7.4.5.]{BJ}. To make sure that the final set $B$ has upper density $\geq \varepsilon$ we have to use finite sets $b_n$ instead of singletons.\\
We construct $B=\bigcup\{b_n \in [\omega]^{<\omega}\such n\in \omega \}$ together with a decreasing sequence $\{B_n \subseteq A \such n\in \omega\}$. We start with $B_0 = A$ and $b_0 = \min(A)$. So, assume we have constructed all sets up to $ B_n$ and $b_n$. Let $\{t_0, \dots ,t_{k-1}\}$ enumerate all subsets of $\bigcup\{b_0,\dots ,b_n\}$. We construct $B_{n+1}$ as a decreasing sequence $B_n =: B^0_{n+1} \supseteq \dots \supseteq B^k_{n+1}=:B_{n+1}$. Let $j<k$ be given.
\begin{enumerate}
\item[Case 1:] There exists $C\subseteq B^j_{n+1}\setminus (\max(t_j)+1)$ such that $(s\cup t_j, C) \in D$. Then put $B^{j+1}_{n+1}:= C$.
\item[Case 2:] Otherwise put $B^{j+1}_{n+1}:= B^{j}_{n+1}$.
\end{enumerate}
Finally put $B_{n+1}:= B^k_{n+1}$. Since $B_{n+1}$ has upper density $\geq \varepsilon$ we can find  $k_{n+1}\in \omega$ such that
\[
\frac{|(\bigcup_{j\leq n} b_j \cup B_{n+1}) \cap k_{n+1}|}{k_{n+1}} \geq \varepsilon-2^{-n-1}
\]
and set $b_{n+1}:=  B_{n+1} \cap k_{n+1}$.
\end{proof} 

\begin{corollary}
The forcing $\mathias^+_\varepsilon$ has Axiom A.
\end{corollary}
\begin{proof}
We take the partial order relations $\leq_n$ as defined in Definition \ref{stronger_n} and also recall the sequences $\langle k^p_n \such n<\omega \rangle$ from Remark \ref{sequence}. The crucial part is to make sure that the requirement (c) from Axiom A (Definition \ref{Axiom_A}) is fulfilled. So, fix an open dense set $D \subseteq \mathias^+_\varepsilon, n< \omega $ and a condition $p=(s,A_p)\in \mathias^+_\varepsilon$. Let $N$ be big enough and $\{ t_i \in [\omega]^{<\omega} \such i<N \}$ enumerate all subsets of $A_p \cap k^p_n$. We define a decreasing sequence $A_p \supseteq  B_0 \supseteq B_1 \supseteq , \dots , \supseteq B_N$ such that for $i<N$:
\begin{align*}
\forall (t,C)\in D \big( (t,C)\leq (s\cup t_i , B_i) \rightarrow (t, B_{i+1}\setminus(\max (t)+1))\in D \big).
\end{align*}
We start with $B_0 := A_p \setminus (\max(t_0)+1)$. To get the sets $B_i$ for $i>0$, simply apply Lemma \ref{quasi_pure_decision} to the condition $(s\cup t_{i-1},B_{i-1})$ and the open dense set $D$.\\
Finally, we set $B := (A\cap k^p_n) \cup B_N$ and \\
$E_p:= \{(t,B\setminus (\max(t)+1)) \such \exists C\in[\omega]^\omega ((t,C)\leq (s,B) \wedge (t,C)\in D) \}$. Then, $(s,B) \leq_n (s,A)$ and  $E_p$ is a countable predense set below $(s,B)$. 
\end{proof}

\begin{proposition}\label{Cohen}
$\mathias^+_\varepsilon$ adds a Cohen real.
\end{proposition}
\begin{proof}
Take $N \in \omega$ such that $1/N < \varepsilon$. Divide $\omega$ into $N+1$ disjoint sets $a_i\subseteq \omega,i<N+1$ of density $1/(N+1)$ i.e., $d^+(a_i)=d^-(a_i)=1/(N+1),i<(N+1).$ Then each set $A\subseteq \omega$ with upper density $\geq 1/N$ cannot be completely contained in a single set $a_i$. Furthermore it must intersect at least two of them infinitely often.  Let $x_{G}$ be the canonical name for the $\mathias^+_\varepsilon$-generic real and $\langle n_k \such k \in \omega \rangle$ enumerate all integers $n$ such that $x_G(n) = 1$. We define a $\mathias^+_\varepsilon$-name $\dot{c}$ via:
\[
\dot{c}(k) := 
\begin{cases}
0 , &\exists i<(N+1) \{n_{2k} , n_{2k+1} \} \subseteq  a_i \\
1 , & \text{ else.}
\end{cases}
\]
We claim that $\dot{c}$ is Cohen. For this purpose, fix $t\in 2^{<\omega},(s,A)\in \mathias ^+_\varepsilon$. W.l.o.g. we can assume that  $|s^{-1}(\{1\})|$ is even. Let $r \in 2^{<\omega}$ maximal such that $(s,A)\Vdash r \trianglelefteq\dot{c}$. We have to find $(s',A')\leq (s,A)$ with the property $(s',A')\Vdash r^\conc t \trianglelefteq\dot{c}$. We construct $(s',A')$ as a decreasing sequence $(s,A) =: (s_0,A_0) \geq (s_1,A_1) \geq \dots \geq (s_{|t|},A_{|t|}) = (s',A')$ such that $|{s_i}^{-1}(\{1\})|+2 = |{s_{i+1}}^{-1}(\{1\})|,i<|t|$ and $(s_i, A_i) \Vdash r^\conc t\restric i \trianglelefteq \dot{c} $. We only carry out the first step of the construction. Take $i<(N+1)$ such that $A_0 \cap a_i$ is infinite. Put $m := \min (A_0 \cap a_{i})$. There are two cases:

\begin{align*}
t(0)=0: &\text{ Define } M := \min ( (A_0  \cap a_{i})\setminus(m + 1) ) \text{ and put } s_1 := {s_0} ^\conc m ^\conc M ,\\
& A_1:= A_0 \setminus (M + 1). \text{ Then }(s_1, A_1)\Vdash r^\conc t(0) \trianglelefteq\dot{c}.\\ 
t(0)=1: &\text{ Define } M := \min(A_0 \setminus((m + 1) \cup a_{i}))\text{ and put } s_1 := {s_0} ^\conc m ^\conc M,\\
& A_1 := A_0 \setminus (M + 1). \text{ Then }(s_1, A_1)\Vdash r^\conc t(0) \trianglelefteq\dot{c}. 
\end{align*}

The rest of the construction is carried out analogously.

\end{proof}
Thus, in contrast to $\mathias$ we get.
\begin{corollary}
$\mathias^+_\varepsilon$ does not satisfy pure decision.
\end{corollary}

\section{Lower density $\geq \varepsilon$}
\label{S3}

In this section we investigate the Mathias forcing with lower density $\geq \varepsilon$. The first observation is.
\begin{observation}
The family of sets $\ideal F_1 :=\{A\subseteq \omega \such d^- (A)=1 \}$ is 
a  filter. 
\end{observation}



So, $\mathias^-_1$ is in fact equivalent to Mathias forcing $\mathias(\ideal F_1)$ with respect to the filter $\ideal F_1$.


Thus, we get:

\begin{fact}
$\mathias^-_1$ satisfies the countable chain condition.
\end{fact}

Now, we will see that the forcing $\mathias^-_\varepsilon$ is a disjoint union of $\sigma$-centred forcings.

\begin{definition}
Let $A,B \subseteq \omega$. The lower density of $B$ with respect to $A$ is defined by:
\[
d_A^-(B):= \liminf_{n \rightarrow \infty} \frac{|A\cap B\cap n|}{|A\cap n|}.
\]
\end{definition} 
\begin{lemma}
Let $B\subseteq A \subseteq \omega, d^-(A)>0$ and $d_A^-(B) < 1$. \\
Then, $d^-(B) <d^-(A).$
\end{lemma}
\begin{proof}
Let $B\subseteq A \subseteq \omega$ be as in the lemma. Then, $d_A^-(B)<1$ implies
\[
 (\exists n_0<\omega) ( \forall k<\omega) (\exists m\geq k) \left( \frac{|B\cap m|}{|A\cap m|} < 1 -2^{-n_0} \right).
\]  
So, we in fact get
\[
( \exists^\infty m) \left( |B\cap m|<  (1 -2^{-n_0})\cdot|A\cap m| \right).
\]
\end{proof}

\begin{corollary}\label{technical_lower_density}
If $d^-(A) = \eps$, then ${\ideal F}(A):= \{B \subseteq A \such d^-(B) = \eps\}$ is a filter.
\end{corollary}

\begin{proposition}\label{sigma_centered}
Let $\varepsilon \in (0,1]$.

 The forcing notion $\mathias^-_\varepsilon$ is equivalent to a disjoint union of $\sigma$-centred forcings.

\end{proposition}
\begin{proof}
First, note that $\ideal D := \{p=(s,A_p)\in \mathias^-_\varepsilon \such d^-(A_p)=\varepsilon  \}$ is an open dense subset of $\mathias^-_\varepsilon$. Fix a maximal antichain $\ideal A \subseteq \ideal D$. For each $p\in \ideal A$ the restriction of $\mathias^-_\varepsilon$ to $p$ is denoted by $\mathias^-_\varepsilon\restric p= \{q \in \mathias^-_\varepsilon \such q\leq p\}$. Then, each $\mathias^-_\varepsilon\restric p$ has the countable chain condition. Indeed, fix $p=(s,A_p)\in \ideal A$. By Corollary \ref{technical_lower_density} we know that the set $\{ B\subseteq A_p \such d^-(B)=\varepsilon \}$ is closed under finite intersections. Thus, any two conditions $q_0,q_1\leq p$ which have the same stem, are compatible.

 Finally,  we can set $\mathias^-_\varepsilon=\bigcup_{p\in \ideal A} \mathias^-_\varepsilon\restric p.$
\end{proof}
\begin{corollary}
$\mathias^-_\varepsilon$ is proper for each $\varepsilon\in(0,1)$.
\end{corollary}
Next, we show that $\mathias^-_\varepsilon$ has antichains of size the continuum, whenever $\varepsilon<1$.

\begin{lemma}
Let $\varepsilon \in (0,1)$. There is a family of sets $\{A_f \such f\in I \}$ such that
\begin{enumerate}
\item $I\subseteq 2^\omega$ has size continuum,
\item $d^-(A_f) \geq \varepsilon,$ for all $f\in I$,
\item $d^-(A_f\cap A_g) < \varepsilon$, for all $f\neq g$ in $I$.
\end{enumerate} 
\end{lemma}
\begin{proof}
Let $\varepsilon\in (0,1)$.
First, we fix a suitable set of indices $I\subseteq 2^\omega$. To this end, let $I$ be any family of functions of size continuum with the property
\begin{align}
\forall f,g \in I ( f\neq g \rightarrow \exists^\infty i,j (f(i)=1 > g(i)\wedge f(j)= 0< g(j))).\label{index_family}
\end{align}
Next, we define for $n\in\omega$ the interval $I_n:=[2^{n+1}, 2^{n+2})\subseteq \omega$. Then, all intervals are pairwise disjoint and each interval $I_n $ has size $2^{n+1}$. Now, we define for each $f\in I$ a function $F_f\in \omega^\omega$ via 
$$F_f(n):= \sum_{i=0}^n f(i)\cdot 2^{n-i}.$$
 Observe, that the functions $F_f$ have the property that for $n\in \omega (F_f(n)\in [ 0 , 2^{n+1}))$, in other words there are $2^{n+1}$ possible values for $F_f(n)$.\\
We define the sets $A_f$ of lower density $\geq \varepsilon$ recursively over $n$ as a disjoint union  $A_f = \bigcup_n A^n_f$, such that $A^n_f\subseteq I_n$. To this end, fix $n$ and let $k$ be the closest natural number  $\leq \varepsilon \cdot 2^{n+1}$, so $k:= \max\{ j \in \omega \such j \leq \varepsilon \cdot 2^{n+1} \}$. We differentiate two cases:\\
If $2^{n+1} + F_f(n) + k - 1 < 2^{n+2} : $ We set 
$$A^n_f := \{2^{n+1} + F_f(n), 2^{n+1} + F_f(n) + 1,\dots,2^{n+1} + F_f(n) + k-1\}$$
Otherwise, there is $j \leq k-1$ such that $2^{n+1} + F_f(n) + j = 2^{n+2}-1$ and we can set 
$$A^n_f := \{2^{n+1}, \dots ,2^{n+1} + k-1-j \} \cup \{2^{n+1} + F_f(n),\dots,2^{n+1} + F_f(n) + j\}$$
Finally we set $A_f := \bigcup A^n_f$.  It follows directly from the definition that each set $A_f$ has density $\varepsilon$ and for $n\in\omega$ we get $(A_f\cap I_n = A^n_f)$.\\
It is left to show that for $f\neq g$ the set $A_f\cap A_g$ has lower density $<\varepsilon$. So, fix $f,g\in I$ and let $i_0$ be minimal such that $f(i_0)\neq g(i_0)$. W.l.o.g assume $f(i_0)=1$ and $g(i_0)=0$, so in particular $F_f(n)>F_g(n)$ for all $n\geq i_0$. By property (\ref{index_family}) of $I$ there is $i_1>i_0$ such that $f(i_1) = 1$ and $g(i_1) =0$. Then, for all $n>i_1$ we have 
$$|F_f(n)-F_g(n)| \geq 2^{n-i_0} + 2^{n-i_1} - (\sum_{i>i_0}^{i_1-1}2^{n-i}+ 2^{n-i_1} - 1) > 2^{n-i_1+1}.$$
\begin{claim}
There is a strictly positive constant $c\in (0,1)$ such that for all $n\in \omega $
$$\frac{|A_f\cap A_g \cap I_n|}{|I_n|} \leq \varepsilon - c.$$ 
\end{claim}
Clearly, the claim above implies $d^-(A_f\cap A_g)<\varepsilon$. So, fix $n\in \omega$. Since we are only interested in the size of $A_f\cap A_g \cap I_n$, we might assume, by a transformation argument, that $g(i)=0,i<n$ and $f(0) = 0$. Thus, $A^n_g$ consists of the first $k$ elements of $I_n$ i.e., $\{2^{n+1}, \dots ,2^{n+1} + k-1 \}$ and $F_f(n)-F_g(n) = F_f(n) < 2^{n}$. Again, we have to differentiate two cases:\\
If $F_f(n)+2^{n+1} + k-1 < 2^{n+2}$: Then,
\[
\frac{|A^n_f \cap A^n_g\cap I_n|}{2^{n+1}}\leq \frac{k-2^{n-i_1+1}}{2^{n+1}} \leq \varepsilon - \frac{1}{2^{i_1}}.
\]
Otherwise $F_f(n)+2^{n+1} + k-1 \geq 2^{n+2}$. In this case we must have $\varepsilon> 1/2$ and we get:
\[
\frac{|A^n_f \cap A^n_g\cap I_n|}{2^{n+1}}\leq  \frac{2^{n+1} - 2(2^{n+1}-k)}{2^{n+1}} \leq 2\cdot \varepsilon - 1.
\]
We can set $c:= \min\{2^{-i_1}, (1-\varepsilon)\}$.
\end{proof}

\begin{corollary}
$\mathias^-_\varepsilon$ has antichains of size continuum for $\varepsilon\in(0,1)$.
\end{corollary}
\begin{proof}
Take the family of sets $\{A_f \such f\in I \}$ as in the lemma above. Then, $\{ (\la  \ra, A_f) \such f\in I \}\subseteq \mathias^-_\varepsilon$ is an antichain of size continuum.
\end{proof}
\begin{corollary}
There is no filter $\ideal F$ such that $\mathias(\ideal F)$ and $\mathias^-_\varepsilon$ are forcing equivalent.
\end{corollary}
\begin{proof}
This follows from the previous corollary, together with the fact, that $\mathias (\ideal F)$ is a $\sigma$-centered forcing for each filter.
\end{proof}
Analog to the upper density case $\mathias^+_\varepsilon$ we get that $\mathias^-_\varepsilon$ adds Cohen reals.
\begin{proposition}
$\mathias^-_\varepsilon$ adds Cohen reals.
\end{proposition}
\begin{proof}
We can repeat the proof of Proposition~\ref{Cohen}. Again we divide $\omega$ into $N+1$ disjoint sets $a_i\subseteq \omega,i<N+1 $ of density $1/(N+1)$, where $N$ is such that $1/N<\varepsilon$. Now to carry out the rest of the construction it enough to see that any set of lower density $\varepsilon$ intersects at least two of the sets $a_i$ infinitely often.
\end{proof}
Let $x_G$ be the generic real added by $\mathias^-_\varepsilon$. Then, $x_G$ has lower density $0$ and upper density $\varepsilon$ in the generic extension. So one might try to use the same recipe as in Lemma \ref{dominating}, where it was proven that $\mathias^+_\varepsilon$ adds dominating reals. However, condition (ii) from the proof is not satisfied and the following question remains open. 
\begin{question}
Does $\mathias^-_\varepsilon$ add dominating reals?
\end{question}
Note, a positive answer to this question for $\varepsilon=1$ would also be a positive answer to \cite{raghavan2017density}[Question 38] from Raghavan.
\section{Upper density $>0$}

In this section we investigate $\mathias^+$ i.e., the forcing consisting of Mathias conditions $p\subseteq 2^{<\omega}$ such that the corresponding  set of splitting levels $A_p$ has strictly positive upper density.

\begin{definition}
The density zero ideal $\ideal Z$ is defined by:
\[
\ideal Z := \{ A\subseteq \omega \such d^+(A)=0) \},
\]
and the corresponding coideal is denoted by $\ideal Z^+=\{A\subseteq \such d^+(A)>0 \}$.
\end{definition}
Observe that for $p\in \mathias$ we have $p\in \mathias^+$ iff $A_p\in \ideal Z^+ $. \\
We show that the forcing $\mathias^+$ is proper. In order to do this, we need the following Lemma (compare \cite[Lemma 9.6.]{Farah2006}).

\begin{lemma}\label{equivalent}
$\mathias^+$ is forcing equivalent to the two step iteration of $\ideal P(\omega)/\ideal Z^+  * \mathias (\dot{\ideal F})$, where $\dot{\ideal F}$  is a $\ideal P(\omega)/\ideal Z^+$-generic filter. 
\end{lemma}
For sake of completeness we sketch a proof.
\begin{proof}
We define a map $i:\mathias^+\longrightarrow P(\omega)/\ideal Z^+ * \mathias (\dot{\ideal F})$ via $i((s,A)):= ([A]_{\ideal Z^+} , (s,A))$, where $[A]_{\ideal Z^+}=\{B\subseteq \omega \such A\Delta B \in \ideal Z \}$ is the corresponding equivalence class.  It is not hard to see that $i$ preserves being stronger and being incompatible. We show that the range of $i$ is dense. To this end, fix a condition $([A]_{\ideal Z^+}, (s,\dot{C}))$. Then $\dot{C}$ is a $ P(\omega)/\ideal Z^+$-name for an infinite subset of $\omega$ such that $[A]_{\ideal Z^+} \Vdash \dot{C} \in \dot{\ideal F}$. This implies $[A]_{\ideal Z^+}  \Vdash A\setminus \dot{C}\in \mathcal{Z}$. So we get $i((s,A)) = ([A]_{\ideal Z^+} , (s,A)) \leq ([A]_{\ideal Z^+} , (s,\dot{C})).$
\end{proof}

\begin{corollary}
$\mathias^+$ is proper.
\end{corollary}

\begin{proof}
Let $\mathsf{Fin}$ denote the ideal of finite subsets of $\omega$. In \cite[Theorem 1.3.]{Farah04analytichausdorff} Farah proved that $P(\omega)/\ideal Z^+$ is forcing equivalent to the two step iteration of $\mathcal{P}(\omega)/\mathsf{Fin} $ and a measure algebra of Maharam character $\mathfrak{c}$ and therefore is proper. So, $\mathias^+$ is a finite iteration of proper forcings.
\end{proof}

\begin{theorem}\label{Cohen_zero}
$\mathias^+$ adds Cohen reals.
\end{theorem}
The Theorem follows from \cite[Lemma 9.8.]{Farah2006} and the fact that the coideal $\ideal Z^+$ is not semiselective (compare \cite[Definition 2.1.]{FarahSemiselective}). However, since we will make use of the explicit construction of the Cohen real, we also give a proof. We freely identify sequences $x\in 2^{\leq \omega}$ with their corresponding sets of natural numbers $x\in[\omega]^{\leq \omega}$ via $x \mapsto \{ n\such x(n)=1\}$.
\begin{proof}
We define maximal antichains $\mathcal{A}_n$ in $( \ideal Z^+, \subseteq ^*)$ as follows: For $n \in \omega$ let $A^i_n:= \{ k\in \omega \such k = i \mod 2^{(n+1)}\}$ and put $\mathcal{A}_n :=\{ A^i_n \such i< 2^{(n+1)}  \}$, e.g.  $\ideal A_0$ consists of the even and odd numbers. Let $x_G \in 2^\omega$ be a $\mathias^+$-generic real. Then, in the generic extension $V[x_G]$, there is for each $n\in \omega$ exactly one $i_n<2^{(n+1)}$ such that $x_G\subseteq^* A^{i_n}_n$. To simplify notations we denote with $A_n$ this unique $A^{i_n}_n$. We define two sequences $\langle n_i \such i\in \omega \rangle$ and $\langle m_i \such i\in \omega\setminus \{0 \} \rangle$ as follows: We start with $n_0:= \min(x_G)$. When $n_i$ is known, we put $m_{i+1}:= \min(x_G\setminus(n_i + 1))$. To define $n_{i+1}$ we differentiate two cases: If $x_G\setminus(m_{i+1} + 1)\subseteq A_{m_{i+1}}$ we put $n_{i+1}:= m_{i+1}$. Otherwise we put $n_{i+1}:= \min(x_G\setminus((m_{i+1} + 1)\cup A_{m_{i+1}})).$ Observe that by definition we have $m_i \leq n_i$.  We put
\[
c(i):= \begin{cases}  0 , &\text{ if } x_G\setminus(m_{i+1} + 1)\subseteq A_{m_{i+1}} \\
 1 , &\text{ else.}
\end{cases}
\]
We show that $c$ is Cohen. For this purpose, fix a condition $(s,A)\in \mathias^+$. By density we can assume that there is $i<\omega$ such that $m_i,n_i$ and  $c\restric i$ are decided by $(s,A)$ but none of the values of $m_{i+1},n_{i+1}$ nor $c(i)$. We must have $\max(s)=n_i$. Fix $j\in 2$. It is enough to find $(t,B)\leq (s,A)$ such that $(t,B)\Vdash c(i) = j$ and in addition $(t,B)$ does not decide $m_{i+2},n_{i+2}$ or  $c(i+1)$. Depending on the value of $j$ we distinguish two cases:
\begin{enumerate}
\item[$j=0$:] Let $m:=\min(A), t:= s \cup \{m \}$ and $B:= A_m \cap  A\setminus (m+1)$. 
\item[$j=1$:] Find $m\in A$ such that $A\not\subseteq A_m$ (such an $m$ always exists since $A$ has positive upper density and $A_m$ has density $2^{-(m+1)}$). Next, pick $n\in A\setminus A_m$ and put $t := s\cup \{m \} \cup \{ n\}$. Finally, define $B:= A \setminus (n+1)$.
\end{enumerate}
In both cases we have $\max(t)=n_{i+1}, (t,B) $ decides the value of $c(i)$ to be $j$ but neither does it decide $m_{i+2},n_{i+2}$ and nor $c(i+1)$.
\end{proof}

Next, we show that $\mathias^+$ adds dominating reals. 

\begin{theorem}
$\mathias^+$ adds dominating reals.
\end{theorem}
\begin{proof}
Let $\ideal Z^*=\{A\subseteq \omega \such \omega \setminus A \in \ideal Z \}$ be the dual filter of the density zero ideal. In  \cite[Corollary 3]{HRUSAK2014880} Hru\v{s}ak and Minami showed that $\mathias (\ideal F)$ adds dominating reals, whenever $\ideal F$ is a filter extending $\ideal Z^*$. We use Lemma \ref{equivalent}. Let $\dot{\ideal F}$ denote the $P(\omega)/\ideal Z^+$-generic filter. Then, in the generic extension $V[\dot{\ideal F}]$ the filter $\dot{\ideal F}$ extends $(\ideal Z^*)^V$ and therefore $\mathias (\dot{\ideal F})$ adds dominating reals over $V$. Specifically, $\mathias^+$ adds dominating reals over $V$.
\end{proof}

\section{Positive lower density}
\label{s5}

In this section we show that $\silver^-$ collapses the continuum to $\omega$. We also construct an uncountable antichain in the partial order consisting of sets with strictly positive lower density ordered by inclusion. 

\begin{theorem}\label{collapse} The forcing $\silver^-$ collapses the continuum to $\omega$.
\end{theorem}

We first prove a lemma:
 
 \begin{lemma}
 There is a strictly increasing function $\ell\colon \omega \times \omega \times\omega\cup\{-1 \} \to \omega$, a small natural number $r$ and a $\silver^-$-name $\dot{F}$  for a function $F \colon \omega \times \omega \times \omega  \to 2$ in $V[G]$ with the following property:
 \begin{equation}\label{collapserequ}
 \begin{split}
&( \forall \varrho \in 2^\omega) 
(\forall p \in \silver^-) 
(\forall k \in \omega \setminus \{0\})(\forall k_0 \in \omega)\\
&
\Bigl( \bigl(d^-(A_p) \geq \frac{2}{k} \wedge \forall n \geq k_0 
\frac{|A_p \cap [\ell(k,k_0,n),\ell(k,k_0,n+1))|}{\ell(k,k_0,n+1)-\ell(k,k_0,n)} \geq \frac{1}{k}\bigr) \rightarrow \\
& (\exists q_\varrho \leq p )\bigl( d^-(A_{q_\varrho}) \geq \frac{1}{r k} \wedge 
q_\varrho \Vdash (\forall n \in \omega) \dot{F}(k,k_0,n) = \varrho(n)\bigr) \Bigr)
 \end{split}
\end{equation}
\end{lemma}

\begin{proof} $r$, $z_1$, $z_2$ and $m$ will be arranged later. The function $\ell$ is defined recursively by 
\begin{center}
$\ell(k,k_0,n) = \begin{cases} 0, &\text{ if } n= -1,\\
z_1k(k_0+1) , &\text{ if } n = 0,\\
mk\cdot\ell(k,k_0,n-1), &\text{ if } n>0.
\end{cases}
$
\end{center}
Although the function $\ell$ in fact depends on three variables it makes sense to abbreviate it when the values $k$ and $k_0$ are clear from the context. In such a situation we also write $\ell_n$ instead of $\ell(k,k_0,n)$.\\
We let $x_G$ be the generic branch and define 
\[
f(k,k_0,n) = \mbox{ the closest natural number to }\frac{|x_G^{-1}[\{1\}] \cap \ell_n|}{\ell_n} \cdot \frac{z_1 k}{z_2}.
\]
We define in $V[G]$ the following function:
\begin{equation*}
F(k,k_0,n) = \begin{cases} 0, &\mbox{ if } f(k,k_0,n)  \mbox{ is even;}\\
1, &\mbox{ else.}
\end{cases}
\end{equation*}
So, assume that we are given a condition $p\in \silver^-$ that meets the  premise of the implication (\ref{collapserequ}) for $k$ and $k_0$ i.e., $d^-(A_p) \geq \frac{2}{k}$ and $(\forall n \geq k_0) \frac{|A_p \cap [\ell_n,\ell_{n+1})|}{\ell_{n+1}-\ell_n} \geq \frac{1}{k}$. Let $\varrho \in 2^\omega$ be given. Recall any Silver condition $p$ is uniquely described by a function $f_p \colon  \omega \setminus A_p \to 2$.
By induction on $n< \omega$ we define an increasing sequence of partial functions $f_p = f_{q_{-1}} \subseteq f_{q_0} \subseteq f_{q_1} \subseteq \dots$ such that for all $n \geq 0$
\begin{enumerate}
\item $f_{q_n} \restric \ell_{n-1}=f_{q_{n-1}}\restric \ell_{n-1}$,
\item $f_{q_n} \restric [\ell_{n-1},\ell_n) \supseteq f_{q_{n-1}}\restric [\ell_{n-1},\ell_n),$
\item $f_{q_n} \restric [\ell_n, \infty ) = f_{p} \restric [\ell_{n},\infty),$
\end{enumerate}

From the three conditions above it already follows that each corresponding  tree $q_n$ will be a member of $\silver^-$. Additionally, we make sure that the following holds as well for $n\geq 0$:

\begin{enumerate}
\item[(4)] $\frac{|A_{q_n} \cap [\ell_{n-1},\ell_{n})|}{\ell_{n}-\ell_{n-1}}  \geq \frac{1}{rk}$
\item[(5)] $q_n \force \forall m \leq n F(k,k_0,m) = \varrho(m).$
\end{enumerate}

The conditions $(1) - (4)$ together make sure that the decreasing sequence $\la q_n \such n<\omega \ra$ has a lower bound in $\silver^-$, namely $q:= \bigcap_n q_n$. Condition $(5)$ ensures that $q \force \forall n \in \omega F(k,k_0, n) = \varrho (n)$. Thus, we can set $q_\varrho := q $ and are done.

Now for the step from $n-1$ to $n$:
We have 
\[
\frac{|A_p \cap [\ell_{n-1}, \ell_{n})|}{\ell_{n}-\ell_{n-1}} \geq \frac{1}{k}.
\]
This means we can add at most $ \frac{(r-1)}{r}\cdot \frac{1}{k} \cdot (\ell_{n}-\ell_{n-1})$ elements of $[\ell_{n-1}, \ell_{n}) \cap A_p$ to $\dom(f_{q_n})$ and still make sure that condition $(4)$ is met. We will later choose which of them are mapped to 0 and which are mapped to 1 by $f_{q_n}$.

We define two approximations to $f$ and $F$ respectively, which do not depend on the generic element $x_G$. We let
\[
f(k,k_0,n,p) = \mbox{ the closest natural number to }\frac{|f_p^{-1}[\{1\}] \cap \ell_n|}{\ell_n} \cdot \frac{z_1 k}{z_2},
\]

and
\begin{equation*}
F(k,k_0,n,p) = \begin{cases} 0, &\mbox{ if } f(k,k_0,n,p)  \mbox{ is even;}\\
1, &\mbox{ else.}
\end{cases}
\end{equation*}

Now at most $\frac{z_2}{z_1}\cdot \frac{\ell_n}{k}$ many new values are needed to change $f(k_0, k,n,p)$ to $f(k_0, k,n,q) $ such that the quotient 
\[\frac{|f_p^{-1}[\{1\}] \cap \ell_n|}{\ell_n} \cdot \frac{z_1 k}{z_2} \in \omega
\]
and such that $F(k_0, k,n,q) $ coincides with $\varrho(n)$. However, we have to be careful not to contradict condition $(4)$. Especially, the following must hold:
\begin{equation*}
\frac{z_2}{z_1}\cdot \frac{\ell_n}{k} \leq  \frac{(r-1)}{r} \cdot \frac{1}{k} \cdot (\ell_{n}-\ell_{n-1})= \frac{(r-1)}{r} \cdot \frac{1}{k} \cdot \frac{(mk - 1)}{mk} \cdot \ell_n
\end{equation*}

On the other hand, we need to ensure that $q_n$ decides $F(k,k_0,n)$ to be $F(k_0,k,n,q_n)$. By construction and in particular condition  $(4)$, we get that the amount of digits that are in $\ell_{n} \setminus \dom(f_{q_n})$ is $\frac{\ell_{n}}{rk}$.
Hence we need 
\begin{equation*}
\frac{z_2}{z_1}\cdot\frac{\ell_{n}}{k} > 2 \cdot \frac{\ell_{n}}{rk}
\end{equation*}
Both inequalities are true if we have:
\begin{equation*}
2 \cdot \frac{1}{r} < \frac{z_2}{z_1}\cdot\frac{\ell_{n}}{k} \leq \frac{(r-1)}{r} \cdot \frac{(mk - 1)}{mk}. 
\end{equation*}
We can take, e.g.,  $r=m=4$, $z_1 = 3$ and $z_2 = 2$ and get $\frac{1}{2} < \frac{2}{3} \leq \frac{3}{4} \cdot \frac{15}{16}$.
\end{proof}
  
\begin{proof}[Proof of the Theorem]
Now it is easy to see that 
\[ \Vdash_{\silver^-} (\forall \varrho \in V \cap 2^\omega )(\exists k>0) (\exists k_0) 
(\forall n F(k,k_0,n) = \varrho	(n)).\]
Simply fix $p$ and $\varrho$. Compute $k$ and $k_0$ for $p$ such that the  prerequirement of the implication (\ref{collapserequ}) are fulfilled. This is always possible since $d^-(A_p)>0$. Then construct $q_\varrho \leq p$ as in the lemma above.

Hence, $\Vdash_{\silver^-} (k,k_0) \mapsto F(k,k_0, \cdot)$ is a surjection from $\omega \times \omega $ onto $2^\omega \cap V$.

\end{proof}

\subsection{$\mathias^-$ has large antichains}

The following lemma establishes that below any condition $p\in \mathias^-$ there is an antichain of size continuum.
\begin{proposition}\label{continuum_antichains}
There is a family of sets $\{A_f \subseteq \omega \such  f\in I \}$ such that
\begin{enumerate}[]
\item $I \subseteq 3^\omega$ has size continuum,
\item $d^-(A_f) \geq 1/2 $, for all $f\in I$.
\item $d^-(A_f\cap A_g ) = 0, $ for all $f\neq g $ in $I$. 
\end{enumerate} 
\end{proposition}

\begin{proof}
First, we fix a suitable set of indices $I \subseteq 3^\omega$. Therefore pick any family of functions $I$ of size continuum with the property, that for any two different functions $f,g \in I$ there are infinitely many $n\in \omega$ $(f(n)\neq g(n))$. In order to define the sets $A_f$ we will define three auxiliary sets $B_0, B_1, B_2 \subseteq \omega$ and a sequence $\la k^n_{j} \such n\in\omega,j \in  3 \cup \{-1\} \ra$ such that
\begin{enumerate}[i)]
\item $k^{n}_{-1} < k^n_{0}< k^n_{1}< k^n_2 = k^{n+1}_{-1}$, for $n<\omega$,
\item $\frac{k^n_{j-1}}{k^n_j} < 2^{-n}$, for $j < 3, n < \omega$,
\item $d^-(B_i \cup B_j) \geq 1/2$, for $i\neq j$,
\item $d^-(B_i) = 0$, for $i<3$.
\end{enumerate}
We construct the sets $B_i$ recursively over $n\in \omega$ as a disjoint union of sets $\{B^n_i\}_n$ such that each set $B^n_i$ is a subset of $[k^n_{-1},k^n_2)$. Start by defining $B^0_i:= \emptyset,i<3$ and $k_{j}^0:= j+1, j\in 3\cup\{-1\}$. Now assume we have constructed $B^m_i$ and $k^m_j$ for $m<n,i<3 $ and $j\in 3\cup \{-1\}$. We start with $k^n_j$. Let $k^n_{-1}:= k^{n-1}_2$ and choose $k^n_j,j<3$ big enough such that conditions $i)$ and $ii)$ are fulfilled and additionally the difference $k^n_j - k^n_{j-1}$ is an even natural number for each $j<3$.  We perform three construction steps to define the sets $B^n_i$:
\begin{enumerate}[a)]
\item Divide the interval $[k^n_{-1},k^n_0)$ evenly between the two sets $B^n_1$ and $B^n_2$ and avoid  the set $B^n_0$ entirely i.e., 
\begin{align*}
[k^n_{-1},k^n_0) \cap B^n_0 &:= \emptyset,\\
[k^n_{-1},k^n_0) \cap B^n_1 &:= \{ k^n_{-1},  k^n_{-1} + 2 ,\dots , k^n_{0}-2 \}, \\
[k^n_{-1},k^n_0) \cap B^n_2 &:= \{  k^n_{-1} + 1, k^n_{-1}+3 ,\dots ,  k^n_{0}-1 \}.
\end{align*}
\item Divide the interval $[k^n_{0},k^n_1)$ evenly between the two sets $B^n_0$ and $B^n_2$ and avoid  the set $B^n_1$ entirely i.e., 
\begin{align*}
[k^n_{0},k^n_1) \cap B^n_1 &:= \emptyset,\\
[k^n_{0},k^n_1) \cap B^n_0 &:= \{ k^n_{0},  k^n_{0} + 2 ,\dots , k^n_{1}-2 \}, \\
[k^n_{0},k^n_1) \cap B^n_2 &:= \{  k^n_{0} + 1, k^n_{0}+3 ,\dots ,  k^n_{1}-1 \}.
\end{align*}
\item Divide the interval $[k^n_{1},k^n_2)$ evenly between the two sets $B^n_0$ and $B^n_1$ and avoid  the set $B^n_2$ entirely i.e., 
\begin{align*}
[k^n_{1},k^n_2) \cap B^n_2 &:= \emptyset,\\
[k^n_{1},k^n_2) \cap B^n_0 &:= \{ k^n_{1},  k^n_{1} + 2 ,\dots , k^n_{2}-2 \}, \\
[k^n_{1},k^n_2) \cap B^n_1 &:= \{  k^n_{1} + 1, k^n_{1}+3 ,\dots ,  k^n_{2}-1 \}.
\end{align*}
\end{enumerate}
This completes the construction of the sets $B^n_i$ and we can put $B_i := \bigcup_n B^n_i$. We check that conditions iii) and iv) are fulfilled.  Let $i\neq j$ be given. By construction steps $a) - c)$, we know that $B_i \cup B_j$ selects at least each second natural number of each interval $[k^n_{-1},k^n_2)$. Since the intervals $[k^n_{-1},k^n_2)$ partition $\omega$ we get iii).  Condition iv) follows from 
$$\frac{|B_i \cap k^n_i|}{k^n_i} \leq \frac{k^n_{i-1}}{k^n_i} <2^{-n}.$$
Now, we are in a position to define the sets $A_f,f\in I$. For $n\in \omega $ and $i\in 3$, we let $A^n_i := B^n_{i_0} \cup B^n_{i_1}$, where $i_0$ and $i_1$ are chosen such that $\{i,i_0,i_1\}= 3$. Then, we set $A_f := \bigcup_n A^n_{f(n)}$. We have to verify that the sets $A_f$ satisfy conditions $(2)$ and $(3)$. That each set $A_f$ has lower density $\geq 1/2$ follows directly from condition iii) for the sets $B_i$. So let $f,g\in I$ be two different functions and take $n$ such that $f(n)\neq g(n)$. W.l.o.g. assume $f(n)=0$ and $g(n)=1$. Then, from construction step c) it follows that $B_0 \cap B_2 \cap [k^n_1,k^n_2)=\emptyset$ and thus 
$$\frac{|A_f \cap A_g \cap k^n_2|}{k^n_2} \leq \frac{k^n_1}{k^n_2}< 2^{-n}.$$
Since $f$ and $g$ differ on infinitely many $n$ we get $d^-(A_f\cap A_g)=0$. 
\end{proof}

The results of the previous sections are summarized in the table below. 

\vspace{3mm}
\begin{center}
\begin{tabular}{|c|c|c|c|c|c|c|c|c|c|}\hline
$\mathbb{P}$  & $\mathias^+_\varepsilon$ & $\silver^+_\varepsilon$ & $\mathias^-_1$ & $\mathias^-_\varepsilon$ & $\silver^-_\varepsilon$ & $\mathias^+$ & $\silver^+$ & $\mathias^-$ & $\silver^-$ \\ \hline 
proper & \cmark & \cmark & \cmark & \cmark & & \cmark & & & \xmark \\ \hline 
c.c.c  & \xmark & \xmark & \cmark &  \xmark & \xmark  & \xmark & \xmark  & \xmark & \xmark \\ \hline 
Cohen   & \cmark & \xmark & \cmark & \cmark & & \cmark & & \cmark & \\ \hline 
dominating    & \cmark & \xmark & & & & \cmark & & & \\ \hline 
\end{tabular}
\end{center}

\section{Measurability}\label{measurability}

In this section we compare different notions of measurability. We first establish some notations.
\begin{definition} \label{def:ideal-meas}
Let  $\mathcal{X}$ be a non-empty set and $\poset{P}$  be a tree-forcing defined over $\mathcal{X}^\omega$.
  \begin{myrules}
  \item[(1)]
    A subset $X \subseteq \mathcal{X}^\omega$ is called \emph{$\poset{P}$-nowhere dense,} 
if
\[
(\forall p \in \poset{P})( \exists q \leq p) ([q] \cap X = \emptyset).
\]
We denote the ideal of $\mathbb{P}$-nowhere dense sets with $\mathcal{N}_\mathbb{P}$.
\item[(2)]
  A subset $X \subseteq \mathcal{X}^\omega$ is called \emph{$\mathbb{P}$-meager} if it is included in a countable union of $\mathbb{P}$-nowhere dense sets. We denote the $\sigma$-ideal of $\poset{P}$-meager sets with ${\mathcal I}_{\poset{P}}$.
\item[(3)]
A subset $X \subseteq \mathcal{X}^\omega$ is called \emph{ $\poset{P}$-measurable} if 
\[
(\forall p \in \poset{P})( \exists q \leq p) ([q]\setminus X \in \mathcal{I}_\mathbb{P} \vee [q] \cap X \in \mathcal{I}_\mathbb{P}).
\]
\item[(4)] A family $\Gamma\subseteq \mathcal{P}(\mathcal{X}^\omega)$ is called \emph{well-sorted} if it is closed under continuous pre-images. We abbreviate the sentence ``every set in $\Gamma$ is $\poset P$-measurable'' by $\Gamma(\poset P)$.
\end{myrules}
\end{definition}

We make two useful observations concerning the measurability of a set $X$.
\begin{observation}\label{measure}
Let $\poset P$ be a tree-forcing and $X\subseteq \ideal X^\omega$ a set of reals. 
\begin{enumerate}
\item  If $H$ is $\poset P$-comeager, then $X$ is $\poset P$-measurable if and only if $H\cap X$ is $\poset P$-measurable.
\item If $\ideal I_\poset P = \ideal N_ \poset P$, then $X$ is $\poset P$-measurable if and only if 
\[
(\forall p \in \poset{P})( \exists q \leq p) ([q] \subseteq  X  \vee [q] \cap X = \emptyset).
\]
\end{enumerate}

\end{observation}

For Sacks, Laver, Miller, Silver and Mathias forcing we have $\ideal N_\poset P=\ideal I_\poset P$.  In case of the Silver forcing the usual proof for $\ideal I_\poset V = \ideal N_\poset V$ also works for $\silver^+_\varepsilon$, since it only makes use of fusion sequences. But it is unclear whether we can expect the same for the other versions of Silver forcing. \\
In case of the Mathias forcing however,  we don't have an equality in none of the four versions of the forcing.
\begin{theorem}\label{I and N}
\begin{enumerate}
\item $\mathcal{I}_{\silver^+_\varepsilon}=\mathcal{N}_{\silver^+_\varepsilon}$,
\item $\mathcal{I}_{\poset P}\neq \mathcal{N}_{\poset P}$ for $\poset P \in \{ \mathias^+_\varepsilon,\mathias^-_\varepsilon ,\mathias^+ ,\mathias^-\}$.
\end{enumerate}

\end{theorem}
\begin{proof}
(1) As we mentioned above, the first part of the Theorem is a straightforward generalization of the proof for the usual Silver forcing. The only difference being that one has to use  the partial orderings defined as in Definition \ref{stronger_n} to ensure that fusions exist in $\silver^+_\varepsilon.$ \\
(2) We divide the proof into cases.
\begin{enumerate}
\item[$(\geq \varepsilon)$] Let $\varepsilon\in (0,1]$ be given. The proof  is closely intertwined with the fact that the forcings $\mathias^+_\varepsilon$ and $\mathias^-_\varepsilon$ add Cohen reals.  We quickly recall the construction of the Cohen real from Lemma \ref{Cohen} and define a function $\varphi$ such that the image of the generic real is a Cohen real. Fix $\varepsilon > 0, N \in \omega$ such that $1/N <\varepsilon$ and a partition $\{a_i\}_{i<N+1}$ of $\omega$ such that each set $a_i$ has density $1/(N+1)$. Now let $H:=\{x\in 2^\omega \such \exists^\infty i (x(i)=1)\}$. For $x\in H$ let $\langle n_k^x \such k<\omega\rangle$ enumerate all integers $n$ such that $x(n)=1$. We define a function $\varphi:H\rightarrow 2^\omega$ as follows: \[
\varphi(x)(k) := 
\begin{cases}
 0 , &\exists i < (N + 1) \{ n_{2k}^x , n_{2k+1}^x \} \subseteq  a_i  \\
 1 , &\text{ else.}
\end{cases}
\]
For $n\in\omega$ let $M_n:= \{x\in H \such \forall k \geq n (\varphi(x)(k) = 0) \} $. It is not hard to see that for $\poset P \in \{ \mathias^+_\varepsilon,\mathias^-_\varepsilon \}$, and each $n$ we have $M_n\in \ideal N_{\poset{P}}$, but $M := \bigcup_n M_n\not\in \ideal N_{\poset P}$.
\item[$(>0)$] In this case the meager set which is not nowhere dense lies directly at hand.
\begin{claim}
The set $\ideal Z^+$ is $\poset P$-meager but not $\poset P$-nowhere dense, for $\poset P \in \{\mathias^+,\mathias^- \}$. 
\end{claim}
The proof works for both forcings analogously. We only check it for $\mathias^+$ explicitly. For $n\in \omega$ we define the sets $N_n:= \{ x\in 2^\omega \such d^+(x) \geq 1/n \}$. Let $n\in \omega$ and $p \in\mathias^+$ be fixed. We can easily find $q\leq p$ such that $d^+(A_q)< 1/n$. Such a condition $q$ satisfies the property $\forall x \in [q] (d^+(x)<1/n)$ and in particular $[q]\cap N_n = \emptyset.$ This proves that each set  $N_n$ is $\mathias^+$-nowhere dense and so $\ideal Z^+= \bigcup_n N_n$ is  $\mathias^+$-meager. To see that $\ideal Z^+$ cannot be $\mathias^+$-nowhere dense it is enough to check that each condition $p=(s,A_p) \in \mathias^+$ contains a branch $x$ of positive upper density. Clearly, the rightmost branch (i.e. $x(i)=1 \Leftrightarrow i \in s\cup A_p$) fulfills this requirement.
\end{enumerate} 

\end{proof}
Since subsets of $\poset P$-meager sets are $\poset P$-meager as well, $\poset P$-meager and $\poset P$-comeager sets are $\poset P$-measurable we get the following.
\begin{corollary}\label{measurable}
The sets $\ideal Z, \ideal Z^*$ and $\ideal Z^+$ are $\poset P$-measurable, $\poset P\in \{\mathias^+,\mathias^-\}$. 
\end{corollary}

Our next goal is to compare different notions of measurability. 
Let $\poset P$ be any tree-forcing. The statement 
\[
``\Gamma(\poset P)\Rightarrow \Gamma(\poset C) \text{, for each well-sorted family } \Gamma"
\]
is true for various tree-forcings adding Cohen reals e.g. Hechler forcing $\poset D$, Eventually different forcing $\poset E$, a Silver like version of Mathias forcing $\poset T$ introduced in  \cite[Definition 2.1.]{Laguzzi_Stuber} and the full-splitting Miller forcing $\poset {FM}$  \cite[Definition 1.1.]{khomskii_laguzzi}. In fact, it appears reasonable enough to ask.
\begin{question}\label{adding Cohens?}
Does each tree-forcing $\poset P$ adding a Cohen real, necessarily satisfy $\Gamma (\poset P) \Rightarrow \Gamma(\poset C)$, for each well sorted family $\Gamma$?
\end{question}
A partial answer to this question was given in \cite[Proposition 3.1.]{Laguzzi_Stuber}. We restate the proposition in a slightly modified version, that will allow us to generalize the result later on. 
  
\begin{proposition}\label{gamma p implies gamma c}
Let $\mathcal{X}$ be a set of size $\leq \omega$ and $\mathbb{P} $ be tree-forcing  defined on $\mathcal{X}^{<\omega}$. Equip $\mathcal{X}$ with the discrete topology and $\mathcal{X}^\omega $ with the product topology. Let $\varphi^*: \mathcal{X}^{<\omega} \rightarrow 2^{<\omega}$ and $\varphi: \mathcal{X}^\omega \rightarrow 2^\omega$ be two mappings that satisfy the following conditions:
\begin{enumerate}
\item $\varphi^*$ is order preserving and $\varphi^*(\langle\rangle) = \langle \rangle$,
\item $\varphi$ is continuous,
\item $(\forall p\in \poset P) \varphi[p] $ is  open dense in $[\varphi^*(\stem(p))]$,
\item $(\forall p \in \mathbb{P}) \forall  t \trianglerighteq \varphi^* (\stem(p)) \exists p'\leq p (\varphi^* (\stem(p')) \trianglerighteq t).$
\end{enumerate}
 

Then $\Gamma(\mathbb{P} )\Rightarrow \Gamma(\mathbb{C})$, for each well-sorted family.
\end{proposition}

To understand why Proposition \ref{gamma p implies gamma c} is a partial answer to Question \ref{adding Cohens?}, it should be noted that given $\poset P,\varphi$ as in the proposition and $x_G$ a $\poset P$-generic real, we have that $\varphi(x_G)$ is Cohen.

\begin{theorem}\label{mathais_epsilon_implies_cohen}
$\Gamma(\mathias^+_\varepsilon)\Rightarrow \Gamma(\cohen)$, for each well sorted family $\Gamma$.
\end{theorem}
\begin{proof}
Our aim is to apply Proposition \ref{gamma p implies gamma c}. To this end let $H^*:= \{ s\in 2^{<\omega} \such |s^{-1}(\{ 1 \})| \text{ is even } \}$ denote the set of all finite sequences with an even number of $1's$. For $s\in H^*$ let $\langle n_k \such k < |s^{-1}(\{ 1 \})| \rangle$  enumerate $s^{-1}(\{ 1 \})$. Observe that $H^*$ is dense in $2^{<\omega}$. Now fix $N\in \omega $ and a partition $\{ a_i \}_{i<(N+1)}$ of $\omega$ such that $\varepsilon > 1/(N+1) $ and such that each set $a_i$ has density $1/(N+1)$. We define the function $\varphi^*$ as expected. Let $k \in \omega$ and $s\in H^*$ be such that $2k < |s^{-1}(\{ 1 \})| $. We define:
\[
\varphi^*(s)(k):=  
\begin{cases}
0, &\exists i<(N+1) (\{n_{2k} , n_{2k+1}\}\subseteq a_i)  \\
1, &\text{ else.}
\end{cases}
\]
 Let $H$ and $\varphi: H \rightarrow 2^{<\omega}$ be defined as in the proof of Theorem \ref{I and N} (2). Now, given $x\in H$  let $\langle n_k^x \such k<\omega\rangle$ enumerate $x^{-1}\{1 \}$. Then, $\varphi^*(x\restric n_{2k}^x)$ is defined for each $k$ and  $\varphi(x) = \bigcup_k \varphi^*(x \restric n_{2k}^x)$.\\
Note that the same argument used in the proof to show that $\mathias^+_\varepsilon$ adds Cohen reals, gives us $\varphi[p]= [\varphi^*(\stem(p))]$. Especially, condition (3) from the proposition is fulfilled. It is straightforward to check that the other three conditions are satisfied as well and since the set $H$ is $\mathias^+_\varepsilon$-comeager, we can apply the proposition. 
\end{proof}

\section{ $\mathias^+$-measurability}\label{s6}

In this section we examine $\mathias^+$-measurability and give a generalization of Proposition \ref{gamma p implies gamma c}.\\
By Theorem \ref{Cohen_zero} we know that $\mathias^+$
adds Cohen reals. So, it seems reasonable enough to try the same method used in Section \ref{measurability} to prove $\Gamma(\mathias^+_\varepsilon)\Rightarrow \Gamma(\mathbb{C})$, with the forcing $\mathias^+$. However, in doing so one encounters the following problem. The coding from Theorem \ref{Cohen_zero} used to generate the Cohen real, uses information of the \emph{whole} condition and does not depend solely on the stem. To make it clear what we mean, we quickly explain how using the proof of Theorem \ref{Cohen_zero} one gets a coding function defined on a $\mathias^+$-comeager set.
\subsection{Construction of the coding function $\varphi$} Recall that we defined in the proof of Theorem \ref{Cohen_zero} families $\ideal A_n:= \{ A^i_n \such i<2^{(n+1)} \}$, where $A^i_n := \{ k\in \omega \such k = i \mod 2^{(n+1)} \}$. Each $\ideal A_n$ is a maximal antichain in $(\ideal Z^+,\subseteq^*)$ and each set $A^i_n$ has density $2^{-(n+1)}$. Now we define $D_n := \{ A\in [\omega]^\omega \such \exists i< 2^{-(n+1)} (A\subseteq ^* A^i_n) \}$. Then for each $n\in\omega$ and $ (s,A)\in \mathias(\ideal	Z^+)$ there is $B\in D_n$ such that $(s,B) \leq (s,A)$. This already implies that the set $D:= \bigcap D_n$ is $\mathias(\ideal	Z^+)$-comeager and thus it is enough to find a suitable coding function $\varphi$ defined on $D$ instead of $2^\omega$.  First, note that for $x\in D$ we also must have the property $(\forall  n \exists ! i_n (x\subseteq^* A^{i_n}_n))$ and we abbreviate $A^{i_n}_n$ with $A_n$. So, we can define for $x\in D$  two sequences $\langle n_i \such i<\omega \rangle$ and $\langle m_i \such i<\omega\setminus \{ 0 \} \rangle$ as in the proof of Theorem \ref{Cohen_zero}. Finally, $\varphi:D \rightarrow 2^{\omega}$ is defined as
\[
\varphi(x)(i):= 
\begin{cases}
0 , &\text{ if } x\setminus (m_{i+1} + 1)\subseteq A_{m_{i+1}} \\
1 , &\text{ else. }
\end{cases}
\]
This definition of the coding function $\varphi$ seems promising. However, when one takes a close look at possible candidates for a corresponding $\varphi^*$, it becomes clear that $\varphi^*$ cannot depend solely on the stem of a condition. Especially, we cannot apply Proposition \ref{gamma p implies gamma c} as it is stated in Section \ref{measurability}. We have to find a generalization like the following. 

\begin{proposition}\label{gamma p implies gamma q}
Let $\mathcal{X},\mathcal{Y}$ be sets of size $\leq \omega$, $\mathbb{P}, \mathbb{Q}$ be tree-forcings  defined on $\mathcal{X}^{<\omega},\mathcal{Y}^{<\omega}$ respectively. Equip $\mathcal{X},\mathcal{Y}$ with the discrete topology and $\mathcal{X}^\omega , \mathcal{Y}^\omega$ with the product topology. Let $\varphi^*: \mathbb{P} \rightarrow \mathbb{Q}$ and $\varphi: \mathcal{X}^\omega \rightarrow \mathcal{Y}^\omega$ be two mappings that satisfy the following conditions:
\begin{enumerate}
\item $\varphi^*$ is order preserving and $\varphi^*(1_\poset P) = 1_\poset Q$,
\item $\varphi$ is continuous,
\item $\forall p\in \poset P \varphi[p] $ is $\poset Q$-open dense in $[\varphi^*(p)]$,
\item $\forall p \in \mathbb{P} \forall  q \leq \varphi^* (p) \exists p'\leq p (\varphi^* (p') \leq q).$
\end{enumerate}
 Then $\Gamma(\mathbb{P} )\Rightarrow\Gamma(\mathbb{Q})$, for each well-sorted family $\Gamma$.
\end{proposition}

The key difference is that $\varphi^*$ is a map from $\poset P$ to $\poset Q$, instead of being defined for finite sequences.\\
Before we turn to the proof of the proposition we investigate further if we might apply it to $\mathias^+$ and $\mathbb{C}$. To this end we already have defined a $\mathias^+$-comeager set $D$ and a coding function $\varphi:D\rightarrow 2^\omega$ satisfying $\varphi(x_G)$ is Cohen, where $x_G$ is $\mathias^+$-generic. We want to define $\varphi^*:\mathias^+\rightarrow \cohen$. Let $\dot{x}_G$ be the canonical name for the $\mathias^+$-generic real. For $p\in \mathias^+$ let $r_p\in 2^{<\omega}$ be maximal such that $p \Vdash \varphi (\dot{x}_G) \trianglerighteq r_p$ and put $\varphi^*(p):= r_p$.
Observe that we have the following  properties: 
\begin{itemize}
\item $\varphi^*$ is order preserving and $\varphi^*((\langle\rangle,\omega))=\langle\rangle$,
\item $\varphi$ is not continuous,
\item It follows from the proof that $\varphi(\dot{x}_G)$ is Cohen, 
that conditions (3) and (4) of Proposition \ref{gamma p implies gamma q} are satisfied.
\end{itemize}
This shows that we almost get $\Gamma(\mathias^+)\Rightarrow \Gamma(\cohen)$, for any well-sorted family $\Gamma$. In fact, the only time we need $\varphi$ to be continuous is to ensure that the pre-image of a regular set $Y \in \Gamma$ is again regular. So, if we change the requirement of $\Gamma$ of being well-sorted and instead assume that the family of sets $\Gamma$  is closed under pre-images of the $\varphi$ constructed above we get:
\begin{corollary}\label{all_upper_all_baire}
Let $\varphi$ be defined as in the beginning of this section and $\Gamma$ be a family of sets closed under pre-images of $\varphi$. Then  $\Gamma(\mathias^+)\Rightarrow \Gamma(\cohen)$ holds.
\end{corollary}

Now we prove Proposition \ref{gamma p implies gamma q}. The key step is the following lemma.

\begin{lemma} \label{main-lemma}
Let $\mathbb{P},\mathbb{Q},\varphi,\varphi^*$ be as in the Proposition and $Y\subseteq \mathcal{Y}^{\omega}$. Define $X := \varphi^{-1}[Y]$. Assume there is $p \in \mathbb{P}$ such that $X \cap [p]$ is $\mathbb{P}$-comeager in $[p]$. Then $Y \cap [\varphi^*(p)]$ is $\mathbb{Q}$-comeager in  $[\varphi^*(p)]$.
\end{lemma}
\begin{proof}
We are assuming $X \cap [p]$ is $\mathbb{P}$-comeager, for some $p \in \mathbb{P}$. This implies that there is a collection $\{A_ n \such n<\omega \wedge A_n \text{ is } \mathbb{P}\text{-open dense in }[p]\}$ such that $\bigcap_n A_n \subseteq [p]\cap X$. W.l.o.g. assume $A_n \supseteq A_{n+1}$, for all $n$. Let $q = \varphi^* (p)$. We want to show that $\varphi [X] \cap [q] = Y \cap [q]$ is $\mathbb{Q}$-comeager in $[q]$ i.e., we want to find $\{B_n \such n<\omega \}$ $\mathbb{Q}$-open dense sets in $[q]$ such that $\bigcap_n B_n \subseteq Y \cap [q]$. Given $\sigma \in \mathfrak{c}^{<\omega}$ we recursively define on the length of $\sigma$ a set $\{p_\sigma \such \sigma \in \mathfrak{c}^{<\omega}  \} \subseteq \mathbb{P}$ with the following properties:
\begin{description}
\item[1.] $p_{\langle\rangle} = p$,
\item[2.] $\forall \sigma \in \mathfrak{c}^{<\omega} \; \bigcup_{i} [\varphi^* (p_{\sigma^\conc i })]$ is $\mathbb{Q}$-open dense in $[\varphi^* (p_{\sigma})]$,
\item[3.] $\forall \sigma \in \mathfrak{c}^{<\omega} \forall i \in \omega \; ([p_{\sigma^\conc i}] \subseteq \bigcap_{k\leq |\sigma|} A_{k} \wedge p_{\sigma^\conc i} \leq p_{\sigma})$.
\end{description}
Assume we are at step $n$. Fix $\sigma \in \mathfrak{c}^n$ arbitrarily and then put $q_\sigma = \varphi ^* (p_\sigma)$. We first make sure that $2.$ holds. For this purpose, let $\{ q_i \such i<\mathfrak{c} \}$ enumerate all conditions in $\mathbb{Q}$ below $q_\sigma$. By condition (4) from Proposition \ref{gamma p implies gamma q}  we can find $p_i \leq p_\sigma$ such that $\varphi^* (p_i) \leq  q_i$. Since each $A_k$ is $\mathbb{P}$-open dense in $[p]$ we can find for each $i<\mathfrak{c}$ an extension $p_{\sigma^\conc i} \leq p_i$ such that $[p_{\sigma^\conc i} ] \subseteq \bigcap_{k \leq n} A_k$. This ensures that also $3.$ holds.
Finally, we put  $B_{n} := \bigcup\{ \varphi[[p_{\sigma}]]\such \sigma\in \mathfrak{c}^n \}$. \\
We have to check that each set $B_n$ is $\mathbb{Q}$-open dense in $[q]$ and $\bigcap_n B_n \subseteq Y\cap [q]$. So fix $n\in \omega$ and $q'\leq q = \varphi^*(p)$. In the first construction step this $q'$ was enumerated, say by $ i < \mathfrak{c}  $ so $ q'= q_i$ and $ \varphi^* (p_{\langle i \rangle}) \leq q'$. Especially,   $\varphi^*(p_\sigma)\leq q'$, whenever $\sigma \in \mathfrak{c}^n, \sigma(0)=i$. By condition $(3)$ from the Proposition $\varphi[ [p_\sigma] ]$ is $\mathbb{Q}$-open dense in  $[ \varphi^*(p_\sigma)]$. This proves that $B_n$ is $\mathbb{Q}$-open dense in $[q].$ \\
By construction of $B_{n + 1}$ we know $B_n \subseteq \varphi[\bigcap_{k\leq n+1} A_k]$ and hence 
\[
\bigcap B_n \subseteq \varphi [\bigcap_n A_n]\subseteq \varphi[[p] \cap X] \subseteq \varphi [p] \cap Y \subseteq [q] \cap Y.
\]
\end{proof}

\begin{proof}[Proof of the proposition] Let $Y \in \Gamma $ be given and put $X := \varphi^{-1} [Y]$. Then also $X\in\Gamma$, since $\Gamma$ is well-sorted and $\varphi$ is continuous.  We now use the lemma to show that for every $q \in \mathbb{Q}$ there exists $q' \leq q$ such that $Y \cap [q']$ is $\mathbb{Q}$-meager or $Y \cap [q']$ is $\mathbb{Q}$-comeager.  

Observe that by conditions $(1)$ and $(4)$ we get $\varphi^*[\mathbb{P}]$ is dense in $\mathbb{Q}$. Now fix $q \in \mathbb{Q}$ arbitrarily and pick $p \in \mathbb{P}$ such that $\varphi^*(p) \leq q$. By assumption $X$ is $\mathbb{P}$-measurable, and so:
\begin{itemize}
\item in case there exists $p' \leq p$ such that $X \cap [p']$ is $\mathbb{P}$-comeager; put $q ':=  \varphi^*(p')$. By the lemma above, $Y \cap [q']$ is $\mathbb{Q}$-comeager in $ [q']$;
\item in case there exists $p' \leq p$ such that $X \cap [p']$ is $\mathbb{P}$-meager, then apply the lemma above to the complement of $X$, in order to get $Y \cap [q']$ be meager in $ [q']$, with $q':=  \varphi^*(p')$. 
\end{itemize}
\end{proof}

 \begin{corollary}
Let $\mathbb{P},\mathbb{Q},\varphi^*$ and $\varphi$ be as in proposition. Then $\varphi (x_G)$ is $\mathbb{Q}$-generic, where $x_G$ is a $\poset P$-generic real.
\end{corollary}
\begin{proof}[Proof of the Corollary] Fix an $\mathbb{Q}$-open dense set $D\subseteq \mathbb{Q}$. We want to show that the conditions $p\in \mathbb{P}$ such that $p\Vdash \exists q\in D (\varphi(x_G)\in[q])$ is dense in $\mathbb{P}$. To this end, fix a condition $p\in \mathbb{P}$. Then, since $D$ is dense in $\mathbb{Q}$ there is $q'\in D$ below $\varphi^*(p)$.  By condition $(4)$ of Proposition \ref{gamma p implies gamma q} there is $p'\leq p$ such that $\varphi^*(p')\leq q'$. By condition $(3)$ we know that each condition $r\in \mathbb{P}$ forces $\varphi(x_G)$ into $[\varphi^*(r)]$ and hence  $p' \Vdash \varphi(x_G)  \in  [\varphi^* (p')]\subseteq [q']$.

\end{proof}

In light of this one might also generalize Question \ref{adding Cohens?} to the following. 

\begin{question}\label{adding_generics?}
Let $\poset P,\poset Q$ be two tree-forcings and assume $\poset P$ adds a $\mathbb{Q}$-generic. Does $\Gamma(\poset P) \Rightarrow \Gamma (\mathbb{Q})$ hold, for each well-sorted family $\Gamma$?
\end{question}

\section{A model for $\Sigma^1_2(\silver^+_\varepsilon) \wedge \neg \Sigma^1_2(\cohen)$}

We construct a model in which the implication $\Gamma(\silver^+_\varepsilon) \Rightarrow \Gamma(\cohen)$ fails for $\Gamma = \Sigma^1_2$.\\ 
We recall that the shortest splitting node extending $s\in 2^{<\omega}$ is denoted by $\splsuc(s)$ (see Definition \ref{1.1} (e)).

\begin{lemma} \label{silver-cohen}
Let $\varepsilon \in (0,1]$ and $p \in \silver^+_\varepsilon$. Let $\bar \varphi: \splitting(p) \rightarrow 2^{<\omega}$ such that $\bar \varphi(\stem(p)):=\langle \rangle$ and for every $t \in \splitting(p)$ and $j \in \{  0,1 \}$, 
$$\bar \varphi(\splsuc(t^\conc \langle j \rangle)):= \bar \varphi(t)^\conc \langle j \rangle.$$ 
Let $\varphi : [p] \rightarrow 2^\omega$ be the expansion of $\bar \varphi$, i.e. for every $x \in [p]$, $\varphi(x):= \bigcup_{n \in \omega} \bar \varphi(t_n)$, where $\la t_n \such n \in \omega \ra$ is a $\trianglelefteq $-increasing sequence of splitting nodes in $p$ such that $x=\bigcup_{n \in \omega} t_n$. 

If $c$ is Cohen generic over $V$, then
\[
V[c] \models \exists p' \in \silver^+_\varepsilon \land p' \subseteq p \land \forall x \in [p'] (\varphi(x) \text{ is Cohen over } V).
\]
\end{lemma}
\begin{proof}
For a finite tree $T\subseteq 2^{<\omega}$ we define the set of terminal nodes $\term(T):= \{s\in T\such \lnot \exists t\in T (s\triangleleft t)\}$. Consider the following forcing $\poset{P}$ consisting of finite trees $T \subseteq 2^{<\omega}$ such that for all $s,t \in T$ the following holds:
\begin{enumerate}
\item If $ s,t \in \term(T)$, then $|s|=|t|$.
\item If $s,t \not\in \term(T)$ and $|s|=|t|$, then $s^\conc i\in T$ iff $t^\conc i\in T$, $i\in 2$.
\end{enumerate}
The partial order $\poset P$ is ordered by end-extension: $T' \leq T$ iff $T' \supseteq T$ and $\forall t \in T'\setminus T \exists s\in \term(T) (s \trianglelefteq t)$. 

Note $\poset{P}$ is countable and non-trivial, thus it is equivalent to Cohen forcing $\cohen$.
Let $p_G := \bigcup G$, where $G$ is $\poset{P}$-generic over $V$. We claim that $p':= \bar \varphi^{-1}" p_G$ satisfies the required properties. It is left to show that:
\begin{enumerate}
\item for every $x \in [p']$ one has $\varphi(x)$ is Cohen generic, i.e., every $y \in [p_G]$ is Cohen generic;
\item $p' \in \silver_\varepsilon^+$.
\end{enumerate}

For proving (1), let $D$ be an open dense subset of $\cohen$ and $T \in \poset{P}$. It is enough to find $T' \leq T$ such that every $t \in \term(T')$ is a member of $D$.

Let $\{ t_j: j < N \}$ enumerate all terminal nodes in $T$ and pick $r_N$, so that for every $j < N$, $(t_j \concat r_N  \in D)$. Then put $T':= \{t \in 2^{<\omega} \such \exists t_j  \in \term(T) (t \trianglelefteq {t_j}^\conc r_N \}$. Hence $T' \leq T$ and $T' \force \forall y \in [p_G] \exists t \in (T' \cap D) (t \triangleleft y)$. Hence we have proven that 
\[
\force_\cohen \forall y \in [p_G] \exists t \in  D (t \triangleleft y),
\]
which means every $y \in [p_G]$ is Cohen generic over $V$.

For proving (2), one has to verify that the resulting set of splitting levels $A_{p'}$ has upper density $\geq \varepsilon$. It is easy to see that given any condition $T\in \mathbb{P}$ and $n\in \omega$ one can always find an end-extension $T_n \leq T$ such that 
$$T_n \force \exists k<\omega \left( \frac{|A_{p'}\cap k|}{k} \geq \varepsilon - 2^{-n} \right).$$
\end{proof}

\begin{proposition} \label{cohen-uppersilver}
Let $\cohen_{\omega_1}$ be an $\omega_1$-product with finite support and let $G$ be $\cohen_{\omega_{1}}$-generic over the constructible universe $L$. Then, for every $\varepsilon \in (0,1]$
\[
L(\mathbb{R})^{L[G]} \models \text{``All }\text{On}^\omega \text{-definable sets are } (\silver^+_\varepsilon) \text{-measurable''} \land \neg \Sigma^1_2(\cohen).
\]
\end{proposition}

\begin{proof}
The argument is  the same as in the proof of \cite[Proposition 3.7]{BrendleHalbeisenLoewe}. 
Fix $\varepsilon \in (0,1]$. Let for $\alpha\leq\omega_1$ $\mathbb{C}_\alpha$ denote the forcing adding $\alpha$ Cohen reals. 
Let $G$ be $\bC_{\omega_1}$-generic over $L$. 

Let $X$ be an $\text{On}^\omega$-definable set of reals, i.e. $X:= \{x \in 2^\omega: \psi(x,v)   \}$ for a formula $\psi$ with a parameter $v \in \text{On}^\omega$, and let $p \in \silver^+_\varepsilon$.  We aim to find $q \leq p$ such that $[q] \subseteq X$ or $[q] \cap X=\emptyset$. 

We can find $\alpha<\omega_1$ such that $v, p \in L[G \restric \alpha]$. Let $\varphi: [p] \rightarrow 2^\omega$ be as in Lemma \ref{silver-cohen}.
Let $c= G(\alpha)$ be the next Cohen real and write $\bC$ for
the $\alpha$-component of $\bC_{\omega_1}$.
We let
\[
b_0 =  \big \llbracket \llbracket (\psi(\varphi^{-1}(c),v)) \rrbracket_{\cohen_{\alpha}} = \mathbf{0}  \big \rrbracket_\cohen \quad \text{and} \quad 
b_1 =  \big \llbracket \llbracket \psi(\varphi^{-1}(c),v))\rrbracket_{\cohen_{\alpha}}= \mathbf {1}  \big \rrbracket_\cohen.
\]
Then, by $\mathbb{C}$-homogeneity, $b_0 \land b_1 = \mathbf{0}$ and $b_0 \vee b_1= \mathbf{1}$.
Hence, by applying Lemma \ref{silver-cohen}, one can then find $q \leq p$ such that $q \subseteq b_0$ or $q \subseteq b_1$ and for every $x \in [q]$, $\varphi(x)$ is Cohen over $L[G \restric \alpha]$. We claim that $q$ satisfies the required property.
\begin{itemize}
\item Case $q \subseteq b_1$: note for every $x \in [q]$, $\varphi(x)$ is Cohen over $L[G \restric \alpha]$, and so $L[G\restric \alpha][\varphi(x)] \models \llbracket \psi(\varphi^{-1}(\varphi(x)),v) \rrbracket_{\cohen_{\alpha+1}}$. Hence $L[G] \models \forall x \in [q](\psi(x,v))$, which means $L[G] \models [q] \subseteq X$.  
\item Case $q \subseteq b_0$: we argue analogously and get $L[G] \models \forall x \in [q](\neg \psi(x,v))$, which means $L[G] \models [q] \cap X = \emptyset$.
\end{itemize}
Moreover in $L[G]$ it is well-known that $\SSigma^1_2(\cohen)$ fails (see \cite[Theorem 5.8]{brendleloewe1999}  and \cite[6.5.3, p. 313]{BJ}). Hence in $L[G]$ all $\On^\omega$-definable sets are $\silver^+_\varepsilon$-measurable, but there is a $\SSigma^1_2$ set not satisying the Baire property. As a consequence, in particular we obtain $L(\mathbb{R})^{L[G]} \models \Sigma^1_2(\silver^+_\varepsilon) \land \neg \Sigma^1_2(\cohen)$.
\end{proof}

We conclude this section by summarizing our results from Sections 6,7 and 8. 
\vspace*{3mm}
\begin{center}
\begin{tabular}{|c|c|c|c|} \hline
$\mathbb{P}$ & $\mathias^+_\varepsilon$ & $\mathias^+$ & $\silver^+_\varepsilon$ \\ \hline
$\mathcal{I}_\mathbb{P} = \mathcal{N}_\mathbb{P}$ & \xmark & \xmark  & \cmark \\ \hline
$\Gamma(\mathbb{P}) \Rightarrow \Gamma(\mathbb{C})$  & for all $\Gamma$ & $\Gamma = \mathcal{P}(\omega) $ & $\Sigma^1_2(\silver^+_\varepsilon) \not\Rightarrow \Sigma^1_2(\mathbb{C})$ \\ \hline
\end{tabular}
\end{center}

\section{$\spl$ does not have the Sacks property}\label{s7}

In this section we return to one of our questions from \cite{fat}. First, we recall the definition of splitting tree.
\begin{definition} \label{def:spl}
  A tree $p \subseteq 2^{<\omega}$ is called \emph{splitting tree}, short $p \in \spl$, if for every $t \in p$ there is $k \in \omega$ such that for every $n \geq k$ and every $i \in \{0,1 \}$ there is $t' \in p$,
  $t \trianglelefteq t'$ such that $t'(n)=i$.
  We denote the smallest such $k$ by $K_p(t)$. The set $\spl$ is partially ordered by $q \leq_{\spl} p$ iff $q \subseteq p$.
\end{definition}

On $\spl$ we can define a stronger $n$ relation.
\begin{definition}
For two conditions $p,q\in \spl$ we let $q\leq_n p$ if:
\[
q\leq p \wedge \splitting_{\leq n}(p) = \splitting_{\leq n}(q) \wedge \forall t\in \splitting_{\leq n}(p)(K_p(t)= K_q(t)).
\]
\end{definition}
Assume we are given a condition $p\in \spl$ and a node $t \in p$. Then, for each $q \leq_0 p \restric t$ and $n > K_p(t)$ there are two nodes $t_0,t_1 \in \level_n(q)$ such that $t_i (n-1) = i, (i\in 2)$.\\
We make use of this observation in the following definition. 

\begin{definition} Let $p \in \spl$. For each $t \in p$ and each $n > K_p(t)$ we define
\begin{equation}
K(p,t,n) = \min\{ |\Lev_n(q)| \such q \leq_0 p\restric t)\} 
\end{equation}
\end{definition}

The key observation is
\begin{lemma}
For $q \leq p$, $t \in q$, and $n > K_q(t)=K_p(t)$ we have
\begin{equation}
K(q,t,n) \geq K(p,t,n).
\end{equation}
\end{lemma}
\begin{proof} 
Fix $r\leq_0 q\restric t$ such that $|\level_n(r)|= K(q,t,n)$. Then, also $r\leq_0 p\restric t$ since $K_q(t)= K_p(t)$ and thus $|\level_n(r)| \geq K(p,t,n)$.
\end{proof}

However, the observed property is not yet strong enough, since it needs the proviso $K_q(t) = K_p(t)$.
\begin{theorem}\label{not_Sacksthm}
There is $p \in \spl$ such that for any $q\leq p$, $t \in q$,
\begin{equation}\label{ungl1}
\begin{split}
&(\forall^\infty \ell)(\forall t \in \Lev_\ell(q))
(|t^{-1}\{1\}| < \log_2(\ell)^2)
\end{split}
\end{equation}
\end{theorem}

The proof of the theorem consists of the next three lemmata.
\begin{definition}
To improve readability we fix   the following sequence $\la l_n \such n< \omega \ra$ given by 
\[
l_n = 2^n - 1.
\]
\end{definition}

\begin{lemma} \label{lemmanot_Sacks1}
There is a perfect tree $T_0$ such that 
\begin{myrules}
\item[(1)] $
K_{T_0}(\emptyset) =  0$,
\item[(2)] $
(\forall \ell)(|\Lev_\ell(T_0)|  =   \ell+1)$, 
\item[(3)] 
$(\forall n)(\forall \ell \in [\ell_n,\ell_{n+1}) (\forall t \in \Lev_\ell(T_0) ) |t^{-1}\{1\}| \leq n$,
\item[(4)] $(\forall t \in T_0) (\forall n \in \omega
)(t \concat \la \underbrace{0 , \dots, 0}_{n} \ra \in T_0)$.
\end{myrules} 
\end{lemma}
\begin{proof} We let $\Lev_1(T_0) = \{\la 0 \ra ,\la 1\ra \}$. We go by induction on $n$. Suppose 
$\Lev_{\ell_n}(T_0)$ is $\leq_{\rm lex}$- increasingly enumerated by $s_{\ell_n,0}, \dots, s_{\ell_n, \ell_n}$.
Now for $1 \leq  j < \ell_n  $ we let $\Lev_{\ell_n + j }(T_0)$ be increasingly enumerated by $s_{\ell_n + j ,i}$, $i \leq \ell_n + j +1$, as follows
\begin{eqnarray*}
s_{\ell_n + j , i} &=& {s_{\ell_n + j-1,i}} \concat 0, \mbox{ for }  i < 2(j-1),\\
s_{\ell_n +j,2j-2} &=& s_{\ell_n + j-1,2j-2} \concat 0,\\
s_{\ell_n +j,2j-1} &=& s_{\ell_n + j-1,2j-2} \concat 1,\\
s_{\ell_n +j,i+1} &=& s_{\ell_n + j-1,i} \concat 0,
\mbox{ for }  i \in [ 2j-1, \ell_n +j).
\end{eqnarray*}
Then each $t \in \Lev_{\ell_{n+1}}(T_0)$ has exactly one $i \in [\ell_n, \ell_{n+1})$ such that $t(i) =1$.
Items (1), (2) and (4) follow immediately from the construction.
\end{proof}
In many cases a picture of the object helps to understand the cumbersome indices. Here is the specific construction of $T_0$ up to level $l_3 = 7$:
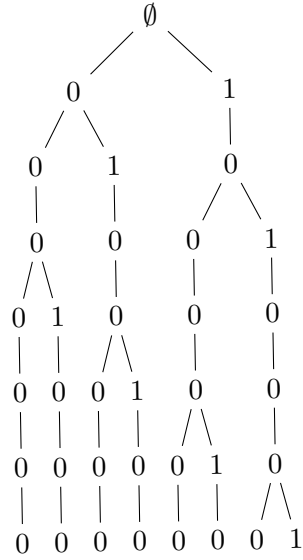
\begin{center}
\begin{turn}{}

\tikzset{every tree node/.style={minimum width=.5em},
         blank/.style={draw=none},
         edge from parent/.style=
         {draw, edge from parent path={(\tikzparentnode) -- (\tikzchildnode)}},
         level distance=1cm}
\begin{tikzpicture}
\Tree
[.$\emptyset$
[.0
	[.0
		[.0
			[.0
				[.0
					[.0
						[.0
						]
					]
				]
			]
			[.1
				[.0
					[.0
						[.0
						]
					]
				]			
			]
		]
	]
	[.1
		[.0
			[.0
				[.0
					[.0
						[.0
						]
					]
				]
				[.1
					[.0
						[.0
						]
					]				
				]
			]
		]
	]
]
[.1
	[.0
		[.0
			[.0			
				[.0
					[.0
						[.0
						]
					]
					[.1
						[.0
						]
					]			
				]
			]
		]
		[.1
			[.0
				[.0
					[.0
						[.0
						]
						[.1
						]
					]
				]
			]
		]
	]
]
]
\end{tikzpicture}
 
\end{turn}
\end{center}
\begin{definition}\label{hoeheres_concat}
For $T \subseteq 2^{<\omega}$ and $s \in 2^{<\omega}$ we let $s \concat T = \{s \concat t \such t \in T\}$.
\end{definition}

\begin{definition}\label{p_from_the_theorem} By induction on $m$ we define $T_m$.
$T_0$ is as in Lemma \ref{lemmanot_Sacks1}.
Now suppose that $T_m$ is defined.
For each $t \in \splitting(T_m) \setminus T_{m-1}$ (for $m=0$, we let $T_{-1} = \{\emptyset\}$)
we choose the minimal $n$ such that $|t| \in [\ell_n, \ell_{n+1})$. Now we let $f(t) = t \concat \la 0,0,0, \dots, 0 \ra$ such that the string of zeros is so long that $|f(t)| = \ell_{n+1}$. Note that $f(t) \in T_m$.
Now we let
\[T_{m+1} = \bigcup \{f(t) \concat T_0 \such t \in 
\splitting(T_m) \setminus T_{m-1}\}.
\]
Note that the $T_0$ is on purpose.
We let $p= \bigcup\{T_m \such m < \omega \}$.
\end{definition}

\begin{lemma} \label{lemmanot_Sacks2}
Let $p$ be as in Definition \ref{p_from_the_theorem}.
Then $p \in \spl$.
\end{lemma}

\begin{proof} Let $t \in \splitting (p)$ and let $m$ be minimal such that 
$t \in T_m$. If $m = -1$, then $t = \stem (T_0) = \emptyset$ and already $K_{T_0}(t) = 0$.
Now assume that $m \geq 0$. Then $f(t) \in T_m$ and $f(t) $ is the stem of $f(t) \concat T_0 \subseteq T_{m+1}$. So $K_{T_{m+1}}(t) = |f(t)|$. In particular $K_p(t)$ is defined and at most $|f(t)|$.
\end{proof}

\begin{lemma} \label{lemmanot_Sacks3}
Let $p$ be as in the definition and $\ell \in \omega$.
For any $t \in \Lev_\ell(p)$, there are at most
$\log_2(\ell)^2$ elements $i < |t|$ such that $t(i) = 1$.
\end{lemma}
\begin{proof}
For each $t \in p$, $|t| \leq \ell_{n+1}$, there is some $j \leq n \leq \log_2(|t|)$ and are $t_i \in T_i$, $i \leq j$, such that
\[t_0 \triangleleft f(t_0) \trianglelefteq t_1 \triangleleft f(t_1)\trianglelefteq \dots  \trianglelefteq
t_j=t \]
and for $i \leq j$ there are $n_i \leq n$ such that
$n_i < n_{i+1}$ for $0 \leq i \leq j$ and
and $|t_i| \in [\ell_{n_i}, \ell_{n_i+1})$. We let $t_{-1} = \emptyset$.
Now by the properties of $T_0$, the number of $k \in [|t_{i-1}|,|t_i|)$ such that $t_i(k)=1$ is bounded by
$n_{i+1} \leq n+1 \leq \log_2(\ell)$. Since also $j \leq \log_2(\ell)$, we have
\[|t^{-1}\{ 1 \}| \leq \log_2(\ell)^2.\]
\end{proof}
Now we turn our attention to the Sacks property and the question whether $\spl$ satisfies this property.
\begin{definition}\hfill
  \begin{enumerate}
    \item $\la S_n \such n < \omega\ra$ is called an \emph{$f$-slalom} if
$S_n \subseteq [\omega]^{f(n)}$.
\item 
A forcing $\bP$ has \emph{the Sacks property} if for any
$f \colon \omega \to \omega \setminus\{\emptyset\}$ such that $\lim_n f(n)= \infty$ and any $\bP$-name $\tau$ for a real and any condition $p$ there is an
$f$-slalom $\la S_n \such n < \omega \ra$ and there is $q \leq p$ such that
\[q \Vdash (\forall^\infty n)(\tau(n) \in S_n).\]
  \end{enumerate}
\end{definition}
\begin{theorem}\label{not_Sacks} At least under a condition $p$ as in Theorem \ref{not_Sacksthm}
the forcing $\spl$ does not satisfy the Sacks property.
\end{theorem}

The proof consists of the following two lemmata.
\begin{lemma}\label{lemmanot_Sacks4} Let $p$ be as in the Theorem \ref{not_Sacksthm}. Then
\[(\forall q \leq p)(\forall t \in \splitting(q)) (\forall^\infty n)\Bigl( K(q,t,\ell_n) \geq \frac{\ell_n}{2\cdot \log_2(\ell_n)^2}\Bigr).
\]
\end{lemma}

\begin{proof} Suppose not and fix a $q \leq p$ and a node $t\in \splitting(q)$ such that
$K_q(t) = k$ and 
\[
(\exists^\infty n) K(q,t,\ell_n) < \frac{\ell_n}{2\cdot \log_2(\ell_n)^2}=: M_n.
\]
We fix such an $n$ such that $\ell_n >2k$.
We fix a slimmest $K_q(t)$ witnessing tree $W \subseteq \Lev_{\leq \ell_{n}}(q\restric t)$ of width $K(q,t,\ell_{n})< M_n$ at level $\ell_n$.
That means $|\level_{\ell_n}(W)|=K(q,t,\ell_{n})$ and on each height $i \in [k, \ell_{n})$ there is at least one $s \in W$ such that $s(i)=1$. Hence there is $s \in \level_{\ell_n}(W)$ such that 
$s$ has at least $\frac{1}{M_n} \cdot (\ell_{n} - k)$ many $i \in [k, \ell_{n})$ such that $s(i) = 1$. We pick such an $s \in W\cap \Lev_{\ell_{n}}(q)$. 
However now $s$ shows 
\begin{equation*}\label{ungl2}\frac{1}{M_n} \cdot (\ell_{n} - k) > \frac{\ell_n}{ 2 \cdot M_n} = \log_2(\ell_{n})^2.
\end{equation*}
This contradicts the property \eqref{ungl1} that says that there are at most $\log_2(\ell_{n})^2$ many $i < |s|$ such that $s(i)=1$. 
 \end{proof}

Now we apply the next lemma with
\[g(n) = \frac{n}{2 \cdot \log_2(n)^2}.
\]

\begin{lemma}\label{unreachable reals 2}
Let $g \colon \omega \to \omega$ be such that $\lim_{n \to \infty} g(n) = \infty$.
 Let $p$ be such that  $(\forall q \leq p)(\forall t \in \splitting(q))(\forall ^\infty n )(K(q,t,\ell_n)  \geq g(\ell_n))$. For any function $f \colon \omega \to \omega \setminus\{\emptyset\}$
  there is a $\spl$-name $\tau$ such that for any $f$-slalom
  $\la S_n \such n < \omega\ra \in V$
  \[p\Vdash_\spl (\exists^\infty n) (\tau(n) \not\in S_n).\] 
\end{lemma}

\begin{proof}
Fix $f \in \omega^\omega \cap V$. 
We give a name $\dot{h_{f}} \in \omega^\omega \cap V^\spl$ so that for every slalom $S \in ([\omega]^{<\omega})^\omega \cap V$, with $|S(n)|\leq f(n)$, we have that $p$ forces that $\dot{h_{f}}$ is not captured by $S$, i.e., 
\[
p \Vdash (\exists n \in \omega )\dot{h_{f}}(n) \notin S(n).
\] 
We let $\dot{x}$ be the name for the $\spl$-generic real,
i.e. the union $\bigcup\{\stem(p) \such p \in G\}$.

Let $\codes \colon 2^{<\omega} \to \omega$
be an injective function.

We fix an increasing subsequence $\la m_{f,i} \such i < \omega\ra\subseteq \la \ell_n \such n<\omega \ra$ such that 
\[
(\forall i)(g(m_{f,i}) > f(i)).
\]
We define
\[
\tau = \dot{h_{f}} := \langle \codes(\dot{x}\restric m_{f,n}) \such  n \in \omega \rangle.
\] 
We aim to show that any $q \leq p$ forces that $\dot{h_{f}}$ cannot be captured by any $f$-slalom in the ground model.
So fix an $f$-slalom $S \in V$ and $q \leq p$.

Let $t = \stem(q)$.
It is enough to find $n \in \omega$, $r \leq q$ such that $r \force \dot{h_{f}}(n) \notin S(n)$. Pick $n_0 \in \omega$ such that for any $n' \geq n_0$ 
\[
K(q,t,{m_{f,n'}})\geq g(m_{f,n'}) > f(n').
\]
Hence $|\level_{m_{f,n_0}}(q\restric t)|\geq g({m_{f,n_0}}) > f(n_0)$. Let $\{t_k \in q\restric t \such |t|=m_{f,n_0} , k < g(m_{f,n_0})\}$ enumerate
the first $g(m_{f,n_0})$ nodes  of  level ${m_{f,n_0}}$  of $q\restric t$.
Then 
\[
q \restric t_k \force \dot{h_{f}}(n_0)= \codes (\dot{x} \restric m_{f,n_0}) = \codes (t_k),
\]
Since $|S(n_0)| \leq f(n_0) < g(m_{f,n_0})$ there is a $t_k$ such that
\[
q \restric t_k \force \codes(\dot{x}\restric m_{f,n_0}) = \codes (t_k) = \dot{h_{f}}(n_0) \notin S(n_0).
\]
So, $r :=q \restric t_k$ is the condition  with the desired property.
\end{proof}

\nothing{
To get an overall result on the failure of the Sacks property one could aim at a positive answer to the following question.

\begin{question}
Is there a diverging function $g$ such that there are densely many $p \in \spl$
with the property  
\[(\forall q \leq p)(\forall t \in \splitting(q)
(\forall^\infty n)( K(q,t,n) \geq g(n))?\]
\end{question}

Since the forcing $\spl$ is $\omega^\omega$-bounding, also an affirmative answer to the following question would suffice for the failure of the Sacks property.
\begin{question}
Are there densely many $p \in \spl$
with the property that there is a diverging function $g_p$ that fulfills  
\[(\forall q \leq p)(\forall t \in \splitting(q)
(\forall^\infty n)( K(q,t,n) \geq g_p(n))?\]
\end{question}
}

As we already mentioned, meanwhile Jonathan Schilhan found a complete negative answer to our question \cite{Schilhan}.

\def\cprime{$'$} \def\germ{\frak} \def\scr{\cal} \ifx\documentclass\undefinedcs
  \def\bf{\fam\bffam\tenbf}\def\rm{\fam0\tenrm}\fi 
  \def\defaultdefine#1#2{\expandafter\ifx\csname#1\endcsname\relax
  \expandafter\def\csname#1\endcsname{#2}\fi} \defaultdefine{Bbb}{\bf}
  \defaultdefine{frak}{\bf} \defaultdefine{=}{\B} 
  \defaultdefine{mathfrak}{\frak} \defaultdefine{mathbb}{\bf}
  \defaultdefine{mathcal}{\cal} \defaultdefine{implies}{\Rightarrow}
  \defaultdefine{beth}{BETH}\defaultdefine{cal}{\bf} \def\bbfI{{\Bbb I}}
  \def\mbox{\hbox} \def\text{\hbox} \def\om{\omega} \def\Cal#1{{\bf #1}}
  \def\pcf{pcf} \defaultdefine{cf}{cf} \defaultdefine{reals}{{\Bbb R}}
  \defaultdefine{real}{{\Bbb R}} \def\restriction{{|}} \def\club{CLUB}
  \def\w{\omega} \def\exist{\exists} \def\se{{\germ se}} \def\bb{{\bf b}}
  \def\equivalence{\equiv} \let\lt< \let\gt>

\nothing{
\bibliographystyle{plain}
\bibliography{../sh/lit,../sh/listb,../sh/lista}

\def\cprime{$'$} \def\germ{\frak} \def\scr{\cal} \ifx\documentclass\undefinedcs
  \def\bf{\fam\bffam\tenbf}\def\rm{\fam0\tenrm}\fi 
  \def\defaultdefine#1#2{\expandafter\ifx\csname#1\endcsname\relax
  \expandafter\def\csname#1\endcsname{#2}\fi} \defaultdefine{Bbb}{\bf}
  \defaultdefine{frak}{\bf} \defaultdefine{=}{\B} 
  \defaultdefine{mathfrak}{\frak} \defaultdefine{mathbb}{\bf}
  \defaultdefine{mathcal}{\cal} \defaultdefine{implies}{\Rightarrow}
  \defaultdefine{beth}{BETH}\defaultdefine{cal}{\bf} \def\bbfI{{\Bbb I}}
  \def\mbox{\hbox} \def\text{\hbox} \def\om{\omega} \def\Cal#1{{\bf #1}}
  \def\pcf{pcf} \defaultdefine{cf}{cf} \defaultdefine{reals}{{\Bbb R}}
  \defaultdefine{real}{{\Bbb R}} \def\restriction{{|}} \def\club{CLUB}
  \def\w{\omega} \def\exist{\exists} \def\se{{\germ se}} \def\bb{{\bf b}}
  \def\equivalence{\equiv} \let\lt< \let\gt>
\begin{thebibliography}{10}

\bibitem{BJ}
Tomek Bartoszy\'{n}ski and Haim Judah.
\newblock {\em {Set Theory, On the Structure of the Real Line}}.
\newblock A K Peters, 1995.

\bibitem{BrendleHalbeisenLoewe}
J\"{o}rg Brendle, Lorenz Halbeisen, and Benedikt L\"{o}we.
\newblock Silver measurability and its relation to other regularity properties.
\newblock {\em Math. Proc. Cambridge Philos. Soc.}, 138(1):135--149, 2005.

\bibitem{brendleloewe1999}
J\"{o}rg Brendle and Benedikt L\"{o}we.
\newblock Solovay-type characterizations for forcing-algebras.
\newblock {\em J. Symbolic Logic}, 64(3):1307--1323, 1999.

\bibitem{Bukovsky_Coplakova}
Lev Bukovsk\'{y} and Eva Copl\'{a}kov\'{a}-Hartov\'{a}.
\newblock Minimal collapsing extensions of models of {${\rm ZFC}$}.
\newblock {\em Ann. Pure Appl. Logic}, 46(3):265--298, 1990.

\bibitem{canjar:mathias}
R.~Michael Canjar.
\newblock Mathias forcing which does not add dominating reals.
\newblock {\em Proc. Amer. Math. Soc.}, 104:1239--1248, 1988.

\bibitem{FarahSemiselective}
Ilijas Farah.
\newblock Semiselective coideals.
\newblock {\em Mathematika}, 45(1):79--103, 1998.

\bibitem{Farah04analytichausdorff}
Ilijas Farah.
\newblock Analytic {H}ausdorff gaps. {II}. {T}he density zero ideal.
\newblock {\em Israel J. Math.}, 154:235--246, 2006.

\bibitem{Farah2006}
Ilijas Farah and Jind\v{r}ich Zapletal.
\newblock Four and more.
\newblock {\em Annals of Pure and Applied Logic}, 140(1):3 -- 39, 2006.
\newblock Cardinal Arithmetic at work: the 8th Midrasha Mathematicae Workshop.

\bibitem{Grigorieff}
Serge Grigorieff.
\newblock Combinatorics of ideals and forcing.
\newblock {\em Ann. Math. Logic}, 3:363--394, 1971.

\bibitem{Halbeisen}
Lorenz Halbeisen.
\newblock {\em Combinatorial Set Theory}.
\newblock Springer, 2012.

\bibitem{HRUSAK2014880}
Michael Hru\v{s}\'{a}k and Hiroaki Minami.
\newblock Mathias-Prikry and Laver-Prikry type forcing.
\newblock {\em Annals of Pure and Applied Logic}, 165(3):880 -- 894, 2014.

\bibitem{khomskii_laguzzi}
Yurii Khomskii and Giorgio Laguzzi.
\newblock Full-splitting {M}iller trees and infinitely often equal reals.
\newblock {\em Ann. Pure Appl. Logic}, 168(8):1491--1506, 2017.

\bibitem{fat}
Giorgio Laguzzi, Heike Mildenberger, and Brendan Stuber-Rousselle.
\newblock Fat splitting trees.
\newblock {\em preprint}, 2020.
\newblock http://home.mathematik.uni-freiburg.de/giorgio/publist.html.

\bibitem{Laguzzi_Stuber}
Giorgio Laguzzi and Brendan Stuber-Rousselle.
\newblock More on trees and {C}ohen reals.
\newblock {\em MLQ Math. Log. Q.}, 66(2):173--181, 2020.

\bibitem{MR4054777}
Paolo Leonetti and Salvatore Tringali.
\newblock On the notions of upper and lower density.
\newblock {\em Proc. Edinb. Math. Soc. (2)}, 63(1):139--167, 2020.

\bibitem{Leonetti_Tringali}
Paolo Leonetti and Salvatore Tringali.
\newblock On the notions of upper and lower density.
\newblock {\em Proc. Edinb. Math. Soc. (2)}, 63(1):139--167, 2020.

\bibitem{Mathias:happy}
Adrian Mathias.
\newblock Happy families.
\newblock {\em Ann, Math. Logic}, 12:59--111, 1977.

\bibitem{PriscoH00}
Carlos A.~Di Prisco and James~M. Henle.
\newblock Doughnuts, floating ordinals, square brackets, and ultraflitters.
\newblock {\em J. Symb. Log.}, 65(1):461--473, 2000.

\bibitem{raghavan2017density}
Dilip Raghavan.
\newblock More on the density zero ideal, 2017.

\bibitem{RoSh:470}
Andrzej Ros{\l}anowski and Saharon Shelah.
\newblock {\em {Norms on Possibilities I: Forcing with Trees and Creatures}},
  volume 141 (no. 671) of {\em Memoirs of the American Mathematical Society}.
\newblock AMS, 1999.

\bibitem{Schilhan}
Jonathan Schilhan.
\newblock Private communication.
\newblock 2020.

\bibitem{Szpilrajn1935}
Edward Szpilrajn.
\newblock Sur une classe de fonctions de m. Sierpi\'nski et la classe
  correspondante d'ensembles.
\newblock {\em Fundamenta Mathematicae}, 24(1):17--34, 1935.

\end{thebibliography}
}

\end{document}